\documentclass{article}

\usepackage{latexsym}
\usepackage[utf8]{inputenc}
\usepackage{amsmath}
\usepackage{amsfonts}
\usepackage{graphicx}
\usepackage{dsfont}
\usepackage{color}
\usepackage{times}
\usepackage{hyperref}
\newcommand{\torus}{{\mathbb{T}}}

\usepackage{theorem}

\theoremheaderfont{\scshape}
\newtheorem{theorem}{Theorem}[section]
\newtheorem{corollary}[theorem]{Corollary}
\newtheorem{lemma}[theorem]{Lemma}
\newtheorem{prop1}[theorem]{Proposition}
\theorembodyfont{\normalfont}

\newtheorem{definition}{Definition}[section]
\newenvironment{proof}{{\bf Proof }}{\hbox{~} \hfill \rule{0.5em}{0.5em}\\}
\numberwithin{equation}{section}

\begin{document}
\title{A minimax method  in  shape and topological derivatives and homogenization: the case of Helmholtz equation}
\date{}          

\maketitle
\centerline{\scshape Mame Gor Ngom \footnote{mamegor.ngom@uadb.edu.sn}}
\medskip
{\footnotesize
 \centerline{Universit\'e Alioune Diop de Bambey, }
  \centerline{  Ecole Doctorale des Sciences et Techniques et Sciences de la Soci\'et\'e. }
   \centerline{Laboratoire de Math\'ematiques de la D\'ecision et d'Analyse Num\'erique}
    \centerline{ (L.M.D.A.N) F.A.S.E.G)/F.S.T. }
}

\medskip
\centerline{\scshape Ibrahima Faye \footnote{ibrahima.faye@uadb.edu.sn},  Diaraf Seck \footnote{diaraf.seck@ucad.edu.sn }}
\medskip
{\footnotesize
 \centerline{Universit\'e Alioune Diop de Bambey, UFR S.A.T.I.C, BP 30 Bambey (S\'en\'egal),}
\centerline{  Ecole Doctorale des Sciences et Techniques et Sciences de la Soci\' et\'e. }
   \centerline{Laboratoire de Math\'ematiques de la D\'ecision et d'Analyse Num\'erique}
\centerline{ (L.M.D.A.N). }
} 

\pagestyle{myheadings}
 \renewcommand{\sectionmark}[1]{\markboth{#1}{}}
\renewcommand{\sectionmark}[1]{\markright{\thesection\ #1}}
\begin{abstract}\noindent 
In this paper, we perform a rigourous version of shape and topological derivatives for optimizations problems under constraint Helmoltz problems. A shape and topological optimization problem  is formulated by introducing cost  functional. We derive first by considering the lagradian method the shape derivative of the functional. It is also proven a topological derivative  with the same approach. An application to several unconstrained shape functions arising from differential geometry are also given.
\end{abstract}

\noindent
\textbf{Keywords:} Minimax method, shape, topological derivatives, averaged Lagrangian, homogenization, two scale convergence, Helmholtz equation.

\section{Introduction}
This paper deals with the notion of shape and topological derivative of functional under Helmholtz equation as constraint. The shape derivative is a differential on some appropriate metric spaces \cite{DelfourMC, DelfourZol} and references there in. The topological derivative introduced by Sokolowski and Zochowski \cite{SokoloZoch} is only a semi differential, this arises from the fact that the tangent space to the underlying metric space of geometries is only a cone. This type of derivative is obtained by expansion methods. In \cite{DelfourDS}, Delfour introduced the notion of differentials and semi-differentials for metric spaces and geometries. The semi differential of the objective function is obtained by introducing a Lagrangian and the adjoint state of the initial problem. This problem is equivalent to the derivative of the minimax of the
Lagrangian with respect to a positive parameter $s$ as it goes to 0. We have inspired by the results of Sturm in his thesis \cite{These-Sturm} and Delfour and Sturm \cite{Delfour2, Delfour-Sturm}, to introduce the new approach for shape and topological derivative. For this, we introduce first the basic notion in differentials calculus for geometries and some methods to derive shape and topological derivative.  Nowadays, some interesting contributions have already been brought in this topic, a lof of question are still open. We derive the minimax of a Lagrangian with respect to a positive parameter $t.$ 
We improve the minimax approach \cite{Delfour1} introduced by K. Sturm \cite{These-Sturm} to find optimal conditions.

\noindent
Let $\Omega$  be a bounded lipschitzien open domain in $\mathbb R^N,\,\,N\geq2$ with a regular boundary and $\eta$ be the solution of the following Helmhottz equation
\begin{equation}\label{eqq0} -\Delta\eta-k^2\eta= f\,\,\text{in}\,\,\Omega,\,\,\eta=0\,\,\text{on}\,\,\partial\Omega.
\end{equation}
Let us consider the functional defined by
\begin{equation}
J(\Omega)=\int_\Omega\vert\nabla\eta-A\vert^2dx+\int_\Omega\vert\eta-\eta_d\vert^2dx
\end{equation}
where $\eta$ is solution to (\ref{eqq0}), $A $ is a given vector and $\eta_d$ a given target.
Our objective, in this paper, is to compute the shape and topological derivative of functional $J(\Omega_s)=j(\Omega_s,\eta_s),$ where the perturbed domain $\Omega_s$ of $\Omega$ is defined by $\Omega_s= T_s(\Omega)$ or $\Omega_s=\Omega\backslash E_r$ depending on the derivative to be calculated. $\eta_s$ is solution to (\ref{eqq0}) posed in $\Omega_s.$ The idea is to use variational techniques due to Sturm \cite{These-Sturm} and Delfour \cite{Delfour1}. Under useful  hypotheses,  Delfour \cite{Delfour1} and Sturm \cite{These-Sturm}  consider that the topological derivative  of the functional  is equal to the s-derivative of the Lagrangian. The calculation of the shape derivative or the topological derivative is possible only if the limit of a certain quantity $R(s)$ exists. In this work, still inspired by their work, we give a sufficient condition for the existence of this limit and therefore the calculation of the shape  or topological derivatives of the functional $J(\Omega)=j(\Omega, \eta_\Omega),$ where $\eta_\Omega$ is solution to the constraint equation. This necessitates the addition of a new assumption on the solution 
$\eta_{\Omega_s}$ in the perturbed domain $\Omega_s.$ 
The rest of the paper is organized as follows: in section 2, we present the main problem and the  result  obtained from the Lagrangian method. This first result is an adaptation of result due to Delfour \cite{Delfour1} and Sturm \cite{These-Sturm}. In section 3, we derive the topological gradient by using also the minimax Lagrangian method. Some useful hypotheses and various cases are also given. We give also some examples where it is not possible to obtain the topological derivative. In the last section, we give our main result based on  asymptotic analysis tools, to ensure topological derivatives by using the two scales convergence in homogenization. %

\section{Problem formulation and shape derivative}


Let $\Omega\subset \mathbb R^N,\,\,N\geq 2$ be a bounded lipschitzian open  with a regular boundary $\partial\Omega$. Let $\eta=\eta(\Omega)$ be the solution of the following Helmholtz problem
\begin{equation}\label{DFEqHelm0}\left\{\begin{array}{ccc}
-\Delta\eta-k^2\eta=f\,\,\text{in}\;\;\Omega,\\
\eta=0\,\,\text{on}\,\,\partial\Omega.\end{array}\right.
\end{equation}
We consider the functional $J$ defined by
\begin{eqnarray}\label{DFfoncHelm1}
J(\Omega)=\int_\Omega\lvert\nabla\eta-A\rvert^2dx+\int_\Omega\lvert\eta-\eta_d\rvert^2dx
\end{eqnarray}
where $\eta$ is the solution of (\ref{DFEqHelm0}), $A\in (L^2(\Omega))^N$ a given vector of $\mathbb{R}^N$ and $\eta_d\in L^2(\Omega)$ a given function.
We want to calculate the shape derivative of the functional $J$.\\
To compute the shape derivative of $J(\Omega)$ by the velocity method, we perturb the bounded open domain $\Omega$ by a family of $T_s$ diffeomorphisms of $\mathbb{R}^N$, $s\geq 0$ generated by a sufficiently regular velocity field $V$. The perturbed domain $\Omega_s$ is then defined by $\Omega_s=T_s(\Omega)$, $0\leq s\leq \tau$, with $\tau>0$. In most cases \cite{SokoZol,MDEL-ZOL} the shape derivative of $J$ in $\Omega$ the direction $V$ direction is defined as
\begin{eqnarray*}
dJ(\Omega;V)=\lim_{s\searrow 0}\;\;\frac{J(\Omega_s)-J(\Omega)}{s}
\end{eqnarray*}
when the limit exists where $\Omega_s=T_s(\Omega)$. The idea here is to use some techniques due to \cite{Delfour1, Delfour2} to compute the shape derivative of the functional $J,$ defined in the perturbed domain $\Omega_s$ by

\begin{eqnarray}\label{DFfoncHelmpert}
J(\Omega_s)=\int_{\Omega_s}\lvert\nabla\eta_s-A\rvert^2dx+\int_{\Omega_s}\lvert\eta_s-\eta_d\rvert^2dx
\end{eqnarray}
where $\eta_s$ is the solution to the problem
\begin{equation}\label{DFEqHelmpert0}\left\{\begin{array}{ccc}
-\Delta\eta_s-k^2\eta_s=f\;\;\text{in}\,\,\Omega_s\\
\qquad\qquad\quad\eta_s=0\;\;\text{on}\,\,\partial\Omega_s,\end{array}\right.
\end{equation}
with $\partial\Omega_s$ is the boundary of $\Omega_s$.\\
The variational formulation of (\ref{DFEqHelmpert0}) is to find $\eta_s\in H^1_0(\Omega_s)$ such that
\begin{eqnarray}\label{DFEqHelmpert}
\int_{\Omega_s}\nabla\eta_s\nabla v'-k^2\eta_s v'dx=\int_{\Omega_s} fv' dx,\;\text{for all}\;v'\in H^1_0(\Omega_s).
\end{eqnarray}
We see clearly that the space of functions $H^1_0(\Omega_s)$ depends on the parameter $s$. To obtain a space independent of $s$, we introduce the following parametrization
\begin{eqnarray}
H^1_0(\Omega_s)=\{\varphi\circ T^{-1}_s,\;\;\varphi\in H^1_0(\Omega) \}.
\end{eqnarray}
Therefore to work in the fixed domain $H^1_0(\Omega)$, we introduce the compositions $\eta^s=\eta_s\circ T_s$ and $v=v'\circ T_s$. The variational formulation of (\ref{DFEqHelmpert}) then amounts to finding $\eta^s\in H^1_0(\Omega)$ such that
\begin{align}
\int_{\Omega_s}\nabla(\eta^s\circ T_s^{-1})\nabla(v\circ T_s^{-1}) -k^2(\eta^s\circ T_s^{-1}) (v\circ T_s^{-1})dx=\int_{\Omega_s} f(v\circ T_s^{-1}) dx,\;\;\forall\;v\in H^1_0(\Omega).
\end{align}
Applying the change of variables formula, we get
\begin{eqnarray*}
\int_{\Omega}\left\lbrace \nabla(\eta^s\circ T_s^{-1})\nabla(v\circ T_s^{-1}) -k^2(\eta^s\circ T_s^{-1}) (v\circ T_s^{-1})\right\rbrace\circ T_sJ_s dx=\int_{\Omega}\left\lbrace f(v\circ T_s^{-1})\right\rbrace\circ T_sJ_s dx.
\end{eqnarray*}
Furthermore, we have
\begin{eqnarray*}
\int_{\Omega}\left\lbrace k^2(\eta^s\circ T_s^{-1}) (v\circ T_s^{-1})\right\rbrace\circ T_sJ_s dx&=&
\int_{\Omega}k^2(\eta^s\circ T_s^{-1}\circ T_s) (v\circ T_s^{-1}\circ T_s)J_s dx\\&=&\int_\Omega k^2\eta^s vJ_sdx
\end{eqnarray*}
\begin{eqnarray*}
\int_{\Omega}\left\lbrace f(v\circ T_s^{-1})\right\rbrace\circ T_sJ_s dx=\int_{\Omega} (f\circ T_s) (v\circ T_s^{-1}\circ T_s)J_s dx=\int_{\Omega} f\circ T_svJ_s dx
\end{eqnarray*}
Let $DT_s(X)$ be the Jacobian matrix of $T_s$ evaluated in $X$, $^*DT_s(X)$ the transpose matrix of $DT_s(X)$, then we have the following proposition

We are therefore able to define the $s$-dependent Lagrangian:
\begin{eqnarray}\label{MG}
L(s,\varphi,\phi)&=&\int_\Omega \left\lbrace B(s)(\nabla\varphi-\nabla\tilde{\eta}^s)\right\rbrace\cdot(\nabla\varphi-\nabla\tilde{\eta}^s)+\lvert\varphi-\eta_d\circ T_s\rvert^2J_sdx+\\
[0.2cm]&+&
\label{MG1}\int_{\Omega}B(s) \nabla \varphi\cdot\nabla\phi - k^2\varphi\phi J_s- f\circ T_s\phi J_s dx
\end{eqnarray}
and defined $j(s)$ as follows
\begin{eqnarray*}
g(s)=\inf_{\varphi\in H^1_0(\Omega)}\sup_{\phi\in H^1_0(\Omega)}L(s,\varphi,\phi);\;\;dg(0)=\lim_{s\searrow 0}\;(g(s)-g(0))/s=dJ(\Omega;V(0)).
\end{eqnarray*}
In (\ref{MG})-(\ref{MG1}), $B(s)$ designate the following formula
$B(s)=J_sDT^{-1}_s(DT^{-1}_s)^{*}$ where $J_s=det DT_s.$
\noindent
For $V\in C^0([0,\tau];C_0^1(\mathbb{R}^N,\mathbb{R}^N))$ and the diffeomorphism $T_s(V)=T_s,$ we have
\begin{eqnarray*}
\frac{d}{ds}T_s(X)=V(s,T_s(X)),\,\;\;T_0(X)=X,\,\;\;\frac{dT_s}{ds}=V(s)\circ T_s,\,\;\;T_0=I.
\end{eqnarray*}
where $V(s)(X)=V(s,X)$ and $I$ the identity matrix in $\mathbb{R}^N$. However
\begin{eqnarray*}
\frac{d}{ds}DT_s=DV(s)\circ T_sDT_s,\,\;\;DT_0=I,\,\;\;\frac{d}{ds}J_s=divV(s)\circ T_sJ_s,\,\;\;J_0=1
\end{eqnarray*}
where $DV(s)$ and $DT_s$ are the Jacobian matrices of $V(s)$ and $T(s)$.\\For $k\geq 1$, $C_0^k(\mathbb{R}^N,\mathbb{R}^N)$ is the space of $k$ times continuously differentiable functions from $\mathbb{R}^N$ to $\mathbb{R}^N$ tending to zero at infinity; for $k=0$, $C_0^0(\mathbb{R}^N,\mathbb{R}^N)$ is the space of continuous functions from $\mathbb{R}^N$ to $\mathbb{R}^N$ tending to zero at infinity.

\noindent
We will use the following notations
\begin{eqnarray*}
V_s=V(s)\circ T_s,\;\;V_s(X)=V(s,T_s(X)),\;\;j(s)=T_s-I,\;\;l(s)(x)=T_s(X)-X.
\end{eqnarray*} In the following, we give the following lemma, which can found in \cite{MDEL-ZOL}.
\begin{lemma}\label{DFdellemme} \item
	Assume that $V\in C^0([0,\tau];C_0^1(\mathbb{R}^N,\mathbb{R}^N))$, $l\in C^0([0,\tau];C_0^1(\mathbb{R}^N,\mathbb{R}^N))$.\\For $\tau>0$ sufficiently small $J_s=\det DT_s=\lvert\det DT_s\rvert=\lvert J_s\rvert$, $0\leq s\leq \tau$ and there are constants $0<\alpha<\beta$ such that
	\begin{eqnarray}
		\forall\;\xi\in\mathbb{R}^N,\;\alpha\lvert\xi\rvert^2\leq A(s)\xi\cdot\xi\leq \beta\lvert\xi\rvert^2\;\;\text{and}\;\;\alpha\leq J_s\leq\beta.
	\end{eqnarray}
	\begin{itemize}
		\item[(i)] As $s$ goes to zero, $V_s\rightarrow V(0)$ in $C_0^1(\mathbb{R}^N,\mathbb{R}^N)$, $DT_s\rightarrow I$ in $C_0^0(\mathbb{R}^N,\mathbb{R}^N)$, $J_s\rightarrow 1$ in $C_0^0(\mathbb{R}^N,\mathbb{R})$
		\begin{eqnarray*}
			\frac{DT_s-I}{s}\;\;\text{is bounded in}\;C_0^0(\mathbb{R}^N,\mathbb{R}^N),\;\;\frac{J_s-1}{s}\;\;\text{is bounded in}\;C_0^0(\mathbb{R}^N).
		\end{eqnarray*}
		\item[(ii)] As $s$ goes to zero
		\begin{eqnarray*}
			B(s)\rightarrow I\;\;\text{in}\;C_0^0(\mathbb{R}^N,\mathbb{R}^N),\;\;\;\frac{B(s)-I}{s}\;\;\text{is bounded in}\;C_0^0(\mathbb{R}^N;\mathbb{R}^N)
		\end{eqnarray*}
		$B'(s)=divV_sI-DV_s-DV_s^*\rightarrow B'(0)=divV(0)-DV(0)-DV(0)^*$ in $C_0^0(\mathbb{R}^N;\mathbb{R}^N)$, \\where $DV_s$ is the Jacobian matrix of $V_s$ and $DV_s^*$ the transpose matrix of $DV_s$.
		\item[(iii)] Given $h\in H^1(\mathbb{R}^N)$, as $s$ goes to zero
		\begin{eqnarray*}
			h\circ T_s\rightarrow h\;\;\text{in}\;\;L^2(\Omega),\;\;\frac{h\circ T_s-h}{s}\;\;\text{is bounded in}\;\;L^2(\Omega)
		\end{eqnarray*}
		\begin{eqnarray*}
			\nabla h\cdot V_s\rightarrow\nabla h\cdot V(0)\;\;\text{in}\;\;L^2(\Omega).
		\end{eqnarray*}
	\end{itemize}
\end{lemma}
\begin{proof} See \cite{MDEL-ZOL}.
\end{proof}
%
In the following, we need to calculate the derivatives of the Lagrangian with respect to the variables $s,\, \varphi$ and $\phi.$ when the limits exist, in oder to get the shape derivative of the funtionnal. In this case we have
\begin{eqnarray*}
d_sL(0,x,y)=\lim_{s\searrow 0}\frac{L(s,x,y)-L(0,x,y)}{s}\\\varphi\in X,\;\;d_xL(s,x,y;\varphi)=\lim_{\theta\searrow 0}\frac{L(s,x+\theta\varphi,y)-L(s,x,y)}{\theta}\\\phi\in Y\;\;d_yL(s,x,y;\phi)=\lim_{\theta\searrow 0}\frac{L(s,x,y+\theta\phi)-L(s,x,y)}{\theta}
\end{eqnarray*}
$s\searrow 0$ and $\theta\searrow 0$ mean that $s$ and $\theta$ go to $0$ by strictly positive values.\\ Thus, we have, since $L(s,x,y)$ is affine en $y$, for all $(s,x)\in [0,\tau]\times X$,
\begin{eqnarray}
\forall\;y,\psi\in H_0^1(\Omega)\;\;d_yL(s,x,y;\psi)=L(s,x,\psi)-L(s,x,0)=d_yG(s,x,0,\psi).
\end{eqnarray}
Recall that the $s$ dependent lagragian is given by
 \begin{eqnarray} L(s,\varphi,\phi)&=&\int_\Omega \left\lbrace B(s)(\nabla\varphi-\nabla\tilde{\eta}^s)\right\rbrace\cdot(\nabla\varphi-\nabla\tilde{\eta}^s)+\lvert\varphi-\eta_d\circ T_s\rvert^2J_sdx+\\
[0.2cm]&+&
\label{exam0}\int_{\Omega}B(s) \nabla \varphi\cdot\nabla\phi - k^2\varphi\phi J_s- f\circ T_s\phi J_s dx\end{eqnarray} 
The set of solutions (states) $x^s$ at $s\geq 0$ is  denoted
\begin{eqnarray}
E(s)=\left\{ x^s\in H_0^1(\Omega),\;\;\forall\;\psi\in H_0^1(\Omega),\;\;d_yL(s,x^s,0;\psi)=0 \right\}.
\end{eqnarray}
The set of solutions $p^s$ at $s\geq 0$ is denoted
\begin{eqnarray}
Y(s,x^s)=\left\{ p^s\in H_0^1(\Omega),\;\;\forall\;\varphi\in H_0^1(\Omega),\;\;d_xL(s,x^s,p^s;\varphi)=0 \right\}.
\end{eqnarray}
Finally the set of minimisers for the minimax is given by
\begin{eqnarray}
X(s)=\left\{ x^s\in H_0^1(\Omega),\;\,g(s)=\inf_{x\in H_0^1(\Omega)}\sup_{y\in H_0^1(\Omega)}L(s,x,y)=\sup_{y\in H_0^1(\Omega)}L(s,x^s,y) \right\}.
\end{eqnarray}
Following the work of Delfour \cite{Delfour1}, we have the following lemma giving the way to get the infimum of the Lagragian.

\begin{lemma} \textnormal{\textbf{(Constrained infimum and minimax)}}\\
We have the following assertions
\begin{itemize}
\item[(i)] 
\begin{align*}
\inf_{x\in H_0^1(\Omega)}\sup_{y\in H_0^1(\Omega)}L(s,x,y)=\inf_{x\in E(s)}L(s,x,0)
\end{align*}
\item[(ii)] The minimax $g(s)=+\infty$ if and only if $E(s)=\emptyset$. And in this case we have $X(s)=X$.
\item[(iii)] If $E(s)\neq\emptyset$, then 
\begin{align*}
X(s)=\left\{x^s\in E(s):\;\;G(s,x^s,0)=\inf_{x\in E(s)}G(s,x,0) \right\}\subset E(s)
\end{align*}
and $g(s)<+\infty$.
\end{itemize}
\end{lemma}
\begin{proof}
 (i) Since the function $y\mapsto G(s,x,y)$ is affine, then for all $(s,x,y)$
	\begin{eqnarray*}
		G(s,x,y)=G(s,x,0)+d_yG(s,x,0;y).
	\end{eqnarray*}
In this case we have
	\begin{eqnarray*}
		\sup_{y\in Y} G(s,x,y)=G(s,x,0)+\sup_{y\in Y} d_yG(s,x,0;y)=
		\begin{cases}
			G(s,x,0)\;\;\text{if}\;\;x\in E(s)\\+\infty\;\;\text{if}\;\;x\in X\backslash E(s)
		\end{cases}
	\end{eqnarray*}
In this case we have
	\begin{eqnarray*}
		\inf_{x\in X}\sup_{y\in Y} G(s,x,y)=\inf_{x\in E(s)} G(s,x,0)
	\end{eqnarray*}
	(ii) If $g(s)=+\infty$, then for all $x\in X$, $\sup_{y\in Y}G(s,x,y)=+\infty$ or quite simply
	\begin{eqnarray*}
		G(s,x,0)+\sup_{y\in Y} d_yG(s,x,0;y)=+\infty
	\end{eqnarray*}
	This leads to $\sup_{y\in Y} d_yG(s,x,0;y)=+\infty$ and then there exists $y\in Y$ such that $d_yG(s,x,0;y)> 0$ and then we have $\forall\,x\in X$, $x\notin E(s)$ and $E(s)=E(s)\cap X=\emptyset$.\\
	Conversely if $E(s)=\emptyset$ we have directly by definition of the infinimum
	\begin{eqnarray*}
		g(s)=\inf_{x\in E(s)} G(s,x,0)=+\infty
	\end{eqnarray*}
	(iii) If $E(s)\neq\emptyset$ and $x^s\in X(s)$ and $x^s\in E(s)$ from the definition of $(i)$
	\begin{eqnarray*}
		G(s,x^s,0)\leq \sup_{y\in Y} G(s,x^s,y)=g(s)=\inf_{x\in E(s)} G(s,x,0)\leq +\infty
	\end{eqnarray*} 
	if $x^s\in E(s)$, $\sup_{y\in Y}d_y G(s,x^s,0;y)=+\infty$ and
	\begin{eqnarray*}
		g(s)=\sup_{y\in Y}d_y G(s,x^s,y)=G(s,x,0)+\sup_{y\in Y}d_y G(s,x,0;y)=+\infty
	\end{eqnarray*}
	This contradicts the fact that $g(s)$ is finite. Therefore $X(s)\subset E(s)$ and
	\begin{eqnarray*}
		G(s,x^s,0)\leq \sup_{y\in Y} G(s,x^s,y)=g(s)=\inf_{x\in E(s)} G(s,x,0)\leq G(s,x^s,0)
	\end{eqnarray*} 
	implies that
	\begin{eqnarray*}
		x^s\in X(s),\;\;g(s)=\inf_{x\in E(s)}G(s,x,0).
	\end{eqnarray*}
	Conversely, if there is $x^s\in E(s)$ such that
	\begin{eqnarray*}
		G(s,x^s,0)=\inf_{x\in E(s)}G(s,x,0).
	\end{eqnarray*}
	So for all $y\in Y$, $d_yG(s,x^s,0,y)=0$, $G(s,x^s,y)=G(s,x^s,0)$ and from the definition (i)
	\begin{eqnarray*}
		\sup_{y\in Y}	G(s,x^s,0)= \inf_{x\in E(s)} G(s,x,0)=\inf_{x\in X}\sup_{y\in Y} G(s,x,y)
	\end{eqnarray*} 
	and $x^s \in X(s)$.
\end{proof}
We also need the following assumptions

\textbf{Hypothesis (H0)}\\ 
\begin{itemize}
\item[(i)]: For all $s\in [0,\tau]$, $x^0\in X(0)$, $x^s\in X(s)$ and $y\in  H_0^1(\Omega)$, the function $\theta\mapsto L(s,x^0+\theta(x^s-x^0),y):[0,1]\rightarrow\mathbb{R}$ is absolutely continuous. This implies that for almost all $\theta$ the derivative exists and is equal to
$d_x L(s,x^0+\theta(x^s-x^0),y;x^s-x^0)$ and it is the integral of its derivative. In particular
\begin{eqnarray*}
L(s,x^s,y)=L(s,x^0,y)+\int_0^1d_x L(s,x^0+\theta(x^s-x^0),y;x^s-x^0)d\theta.
\end{eqnarray*} 
\item[ii)]: For all $s\in [0,\tau]$, $x^0\in X(0)$, $x^s\in X(s)$ and $y\in H_0^1(\Omega)$, $\phi\in H_0^1(\Omega)$ and for almost all $\theta\in[0,1]$, $d_x L(s,x^0+\theta(x^s-x^0),y;\phi)$ exist et the functions $\theta\mapsto L(s,x^0+\theta(x^s-x^0),y;\phi)$ belong to $L^1[0,1];$
\end{itemize}
$\textbf{Hypothesis (H1)}$ for all $s\in [0,\tau]$, $g(s)$ is finite, $X(s)=\{x^s\}$ and $Y(s, x^0,x^s)=\{y^s\}$ are singletons\\
$\textbf{Hypothesis (H2)}$ $d_sL(0,x^0,y^0)$  exists.
\\
$\textbf{Hypothesis(H3)}$ 
The following limit exists
\begin{eqnarray*}
R(x^0,y^0)=\lim_{s\searrow 0}\int_0^1 d_xG\left( s, x^0+\theta(x^s-x^0),p^0;\frac{x^s-x^0}{s}\right) d\theta.
\end{eqnarray*}

We have the following result
\begin{theorem} (\cite{These-Sturm,Sturm2}).

Under assymptions \textbf{(H0), (H1), (H2), (H3)}, $dg(0)$ exists and 
\begin{equation}\label{formule}dg(0)=d_sG(0,x^0,p^0)+R(x^0,p^0).\end{equation}
\end{theorem}
\begin{proof}

Recall that $g(s)=G(s,x^s,y)$ and $g(0)=G(0,x^0,y)$ for each $y\in Y$, then for a standard adjoint state $p^0$ at $s=0$
\begin{eqnarray*}
g(s)-g(0)=G(s,x^s,p^0)-G(s,x^0,p^0)+(G(s,x^0,p^0)-G(0,x^0,p^0))
\end{eqnarray*}
Dividing by $s> 0$, we obtain
\begin{eqnarray*}
\frac{g(s)-g(0)}{s}&=&\frac{G(s,x^s,p^0)-G(s,x^0,p^0)}{s}+\frac{G(s,x^0,p^0)-G(0,x^0,p^0)}{s}\\&=&\int_0^1 d_xG\left( s, x^0+\theta(x^s-x^0),p^0;\frac{x^s-x^0}{s}\right) d\theta\\&+&\frac{G(s,x^0,p^0)-G(0,x^0,p^0)}{s}
\end{eqnarray*}
Going to the limit when $s$ goes to $0$, we obtain
\begin{eqnarray*}
dg(0)&=&\lim_{s\searrow 0}\int_0^1 d_xG\left( s, x^0+\theta(x^s-x^0),p^0;\frac{x^s-x^0}{s}\right) d\theta+d_sG(0,x^0,p^0)\\&=&d_sG(0,x^0,p^0)+R(x^0,p^0).
\end{eqnarray*}
To prove the formula (\ref{formule}), it is sufficient to verify the hypotheses of the theorem of Delfour \cite{Delfour1}, Sturm \cite{These-Sturm,Sturm2}. For that, let us recall the Lagrangian

\begin{eqnarray} L(s,\varphi,\phi)&=&\int_\Omega \left\lbrace B(s)(\nabla\varphi-\nabla\tilde{\eta}^s)\right\rbrace\cdot(\nabla\varphi-\nabla\tilde{\eta}^s)+\lvert\varphi-\eta_d\circ T_s\rvert^2J_sdx+\\
[0.2cm]&+&
\label{exam0}\int_{\Omega}B(s) \nabla \varphi\cdot\nabla\phi - k^2\varphi\phi J_s- f\circ T_s\phi J_s dx\end{eqnarray} 
Then the state equation is given by\\
\begin{equation}\label{exam} \text{find }\,\,x^s\in H_0^1(\Omega) \,\,\text{such that for all}\,\,\psi\in H_0^1(\Omega),\,\,\int_\Omega B(s)\nabla x^s\nabla \varphi-k^2x_s\varphi=\int_\Omega f(T_s)\varphi J_sdx
\end{equation}
The adjoint state $p^s$  is solution to
\begin{equation}\label{exam1}
\text{Find}\,p^s\in H_0^1(\Omega)\,\,\text{such that}\,\forall\,\varphi\in H_0^1(\Omega),\,\,2\int_\Omega\vert \eta^0-\eta_d(Ts)\vert\varphi-\Delta \eta^0\varphi+\int_\Omega B(s)\nabla p^s\nabla \varphi-k^2p^s\varphi=0.
\end{equation}
The bilinear forms (\ref{exam}) and (\ref{exam1}) are coercives, then there exists a unique  solution $\eta^s$ and $(\eta^0, p^0)$ solution (\ref{exam}) and (\ref{exam1}). This allows us to write that, $\forall s\in[0,\tau]$ the spaces
$$X(s)=\{\eta^s\}\,\,\text{and}\,\,Y(0, u^0)=\{p^0\}$$ are singletons, then \textbf{Hypothesis (H1)} is verified.
\\
For the \textbf{Hypothesis (H2)}, the $s$-derivative of the lagragian $L(s,\varphi,\phi)$ is given by
\begin{eqnarray*}
d_sL(s,\varphi,\phi)&=& \int_\Omega \left\lbrace B'(s)(\nabla\varphi-\nabla\tilde{\eta}^s)\right\rbrace\cdot(\nabla\varphi-\nabla\tilde{\eta}^s)-2B(s)(\nabla\varphi-\nabla\tilde{\eta}^s)\cdot\nabla\left( [\nabla\tilde{\eta}\cdot V(s)]\circ T_s\right) dx\\&+&\int_\Omega \lvert\varphi-\eta_d\circ T_s\rvert^2 divV_sJ_s-2(\varphi-\eta_d\circ T_s)\nabla\eta_d\cdot V_sJ_s dx\\
&+&\int_\Omega B(s)\nabla\varphi\cdot\nabla\phi-k^2\varphi\phi divV_sJ_s-\left(f\circ T_sdivV_sJ_s+\nabla f\cdot V_sJ_s\right)\phi dx.
\end{eqnarray*}
Taking $\varphi=\eta^0$ and $\phi=p^0$, we have
\begin{eqnarray*}
d_sL(s,\eta^0,p^0)&=& \int_\Omega \left\lbrace B'(s)(\nabla\eta^0-\nabla\tilde{\eta}^s)\right\rbrace\cdot(\nabla\eta^0-\nabla\tilde{\eta}^s)-2B(s)(\nabla\eta^0-\nabla\tilde{\eta}^s)\cdot\nabla\left( [\nabla\tilde{\eta}\cdot V(s)]\circ T_s\right) dx
\\&+&\int_\Omega \lvert\eta^0-\eta_d\circ T_s\rvert^2 divV_sJ_s-2(\eta^0-\eta_d\circ T_s)\nabla\eta_d\cdot V_sJ_s dx\\
&+&\int_\Omega B'(s)\nabla\eta^0\cdot\nabla p^0-k^2\eta^0p^0 divV_sJ_s-\left(f\circ T_sdivV_sJ_s+\nabla f\cdot V_sJ_s\right)p^0 dx.
\end{eqnarray*}
Using the lemma \ref{DFdellemme}, we get
\begin{eqnarray*}
d_sL(0,\eta^0,p^0)&=&\int_\Omega \left\lbrace B'(0)(\nabla\eta^0-\nabla\tilde{\eta})\right\rbrace\cdot(\nabla\eta^0-\nabla\tilde{\eta})-2\left(\nabla\eta^0-\nabla\tilde{\eta}\right)\cdot\nabla\left( [\nabla\tilde{\eta}\cdot V(0)]\right) dx
\\&+&\int_\Omega \lvert\eta^0-\eta_d\rvert^2 divV(0)-2(\eta^0-\eta_d)\nabla\eta_d\cdot V(0) dx\\
&+&\int_\Omega B'(0)\nabla\eta^0\cdot\nabla p^0-k^2\eta^0p^0 divV(0)-\left(fdivV(0)+\nabla f\cdot V(0)\right)p^0 dx,
\end{eqnarray*}
then \textbf{Hypothesis (H2)} is satisfied.\\ To prove the last hypothesis, we give
$x$-derivative of the lagrangian $L(s,\varphi,\psi)$ with respect to $\varphi$ in the direction $\varphi'$:
\begin{align}
d_{\varphi}L(s,\varphi,\phi;\varphi')=\int_\Omega 2B(s)(\nabla\varphi-\nabla\tilde{\eta}^s)\cdot\nabla\varphi'+2(\varphi-\eta_d\circ T_s)\varphi'J_s+ B(s)\nabla\varphi'\cdot\nabla\phi-\lambda\varphi'\phi J_s.
\end{align}
The derivative of the $s$-dependent Lagrangian with respect to $\phi$ in the direction $\phi$
\begin{align}
d_{\phi}L(s,\varphi,\phi;\phi')=\int_{\Omega}B(s) \nabla \varphi\cdot\nabla\phi' - k^2\varphi\phi' J_s- f\circ T_s\phi' J_s dx.
\end{align}
The state equation at $s\geq 0$ and the adjoint state equation at $s=0$ are 
\begin{align}\label{DFeqetas}
\eta^s\in H^1_0(\Omega),\;\;\forall\;\phi'\in H^1_0(\Omega),\;\;\int_{\Omega}B(s) \nabla\eta^s\cdot\nabla\phi' - k^2\eta^s\phi' J_s- f\circ T_s\phi' J_s dx=0
\end{align}
\begin{align}\label{DFeqad0}
p^0\in H^1_0(\Omega),\;\forall\;\varphi'\in H^1_0(\Omega),\;\int_\Omega 2(\nabla\eta^0-\nabla\tilde{\eta})\cdot\nabla\varphi'+2(\eta^0-\eta_d)\varphi'+ \nabla p^0\cdot\nabla\varphi '-k^2 p^0\varphi' dx=0.
\end{align}
\begin{eqnarray*}
R(s)&=&\int_0^1 d_{\varphi}L\left(s,\eta^0+\theta(\eta^s-\eta^0),p^0;\frac{\eta^s-\eta^0}{s}\right)d\theta\\[0.2cm]
&=&\int_\Omega 2B(s)\left\lbrace\nabla\left(\frac{\eta^s+\eta^0}{2}\right)-\nabla\tilde{\eta}^s\right\rbrace\cdot\nabla\left(\frac{\eta^s-\eta^0}{s}\right)
+2\left(\frac{\eta^s+\eta^0}{2}-\eta_d\circ T_s\right)\left(\frac{\eta^s-\eta^0}{s}\right)J_sdx
\\[0.2cm]&+&\int_\Omega B(s)\nabla p^0\cdot\nabla\left(\frac{\eta^s-\eta^0}{s}\right)-k^2 p^0\left(\frac{\eta^s-\eta^0}{s}\right)J_sdx.
\end{eqnarray*}
Taking $\varphi'=\frac{\eta^s-\eta^0}{s}$ in the adjoint equation (\ref{DFeqad0}), we have
\begin{align*}
\int_\Omega 2\nabla(\eta^0-\tilde{\eta})\nabla\left(\frac{\eta^s-\eta^0}{s}\right)+2(\eta^0-\eta_d)\left(\frac{\eta^s-\eta^0}{s}\right)+ \nabla p^0\cdot\nabla\left(\frac{\eta^s-\eta^0}{s}\right)-k^2 p^0\left(\frac{\eta^s-\eta^0}{s}\right) dx=0.
\end{align*}
This means that
\begin{eqnarray*}
&\,&\int_\Omega 2\left\lbrace \nabla\left(\frac{\eta^s+\eta^0}{2}\right)-\nabla\left(\frac{\eta^s-\eta^0}{2}\right)-\nabla\tilde{\eta}^s+\nabla\tilde{\eta}^s- \nabla\tilde{\eta}\right\rbrace\cdot\nabla\left(\frac{\eta^s-\eta^0}{s}\right)dx +\\[0.2cm]&\,& \int_\Omega 2\left(\frac{\eta^s+\eta^0}{2}-\frac{\eta^s-\eta^0}{2}-\eta_d\circ T_s+\eta_d\circ T_s-\eta_d\right)\left(\frac{\eta^s-\eta^0}{s}\right)dx+\\[0.2cm]&\,& \int_\Omega\nabla p^0\cdot\nabla\left(\frac{\eta^s-\eta^0}{s}\right)-k^2 p^0\left(\frac{\eta^s-\eta^0}{s}\right) dx=0.
\end{eqnarray*}
In other words,
\begin{eqnarray*}
&\,&\int_\Omega 2\left\lbrace \nabla\left(\frac{\eta^s+\eta^0}{2}\right)-\nabla\tilde{\eta}^s   \right\rbrace\cdot\nabla\left(\frac{\eta^s-\eta^0}{s}\right)dx- s\int_\Omega\left\lvert\nabla\left(\frac{\eta^s-\eta^0}{s}\right)\right\rvert^2 dx
+\\&\,& \int_\Omega 2(\eta_d\circ T_s-\eta_d)\left(\frac{\eta^s-\eta^0}{s}\right)dx+\int_\Omega 2(\nabla\tilde{\eta}^s- \nabla\tilde{\eta})\cdot\nabla\left(\frac{\eta^s-\eta^0}{s}\right)dx+
\\[0.2cm]&\,& \int_\Omega 2\left(\frac{\eta^s+\eta^0}{2}-\eta_d\circ T_s\right)\left(\frac{\eta^s-\eta^0}{s}\right)dx - s\int_\Omega\left\lvert\frac{\eta^s-\eta^0}{s}\right\rvert^2 dx 
+\\[0.2cm]&\,& \int_\Omega\nabla p^0\cdot\nabla\left(\frac{\eta^s-\eta^0}{s}\right)-k^2 p^0\left(\frac{\eta^s-\eta^0}{s}\right) dx=0.
\end{eqnarray*}
We can therefore rewrite the expression for $R(s)$ as follows
\begin{eqnarray}\label{R(s)1}
R(s)&=&\int_\Omega 2 \left(\frac{B(s)-I}{s}\right) \left\lbrace \nabla\left(\frac{\eta^s+\eta^0}{2}\right)-\nabla\tilde{\eta}^s   \right\rbrace\nabla\left(\eta^s-\eta^0\right)dx+s\int_\Omega\left\lvert\nabla\left(\frac{\eta^s-\eta^0}{s}\right)\right\rvert^2 dx
+\\&\,& \label{R(s)2}
\int_\Omega 2\left(\frac{J_s-1}{s}\right)\left(\frac{\eta^s+\eta^0}{2}-\eta_d\circ T_s\right)\left(\eta^s-\eta^0\right)dx-\int_\Omega 2(\nabla\tilde{\eta}^s- \nabla\tilde{\eta})\nabla\left(\frac{\eta^s-\eta^0}{s}\right)dx-\\[0.2cm]&\,& \label{R(s)3}
\int_\Omega 2\left(\eta_d\circ T_s-\eta_d\right)\left(\frac{\eta^s-\eta^0}{s}\right)dx + s\int_\Omega\left\lvert\frac{\eta^s-\eta^0}{s}\right\rvert^2 dx 
+\\[0.2cm]&\,& \label{R(s)4}
\int_\Omega \frac{B(s)-I}{s}\nabla p^0\nabla\left(\eta^s-\eta^0\right)-k^2\left(\frac{J_s-1}{s}\right) p^0\left(\eta^s-\eta^0\right) dx.
\end{eqnarray}
Thus the following inequality is verified

\begin{eqnarray*}
R(s)&\leq& 2 \left\lVert\frac{B(s)-I}{s}\right\rVert \left\lVert \nabla\left(\frac{\eta^s+\eta^0}{2}\right)-\nabla\tilde{\eta}^s   \right\rVert \left\lVert\nabla\left(\eta^s-\eta^0\right)\right\rVert +s\left\lVert\nabla\left(\frac{\eta^s-\eta^0}{s}\right)\right\rVert^2
+\\&\,&  2\left\lVert\frac{J_s-1}{s}\right\rVert\left\lVert\frac{\eta^s+\eta^0}{2}-\eta_d\circ T_s\right\rVert\left\lVert\eta^s-\eta^0\right\rVert+ 2\left\lVert\nabla\tilde{\eta}^s- \nabla\tilde{\eta}\right\rVert\left\lVert\nabla\left(\frac{\eta^s-\eta^0}{s}\right)\right\rVert+
\\[0.2cm]&\,&  2\left\lVert\frac{\eta_d\circ T_s-\eta_d}{s}\right\rVert\left\lVert\eta^s-\eta^0\right\rVert + s\left\lVert\frac{\eta^s-\eta^0}{s}\right\rVert^2  
+\\[0.2cm]&\;& \;\;\, \left\lVert\frac{B(s)-I}{s}\right\rVert\lVert\nabla p^0\rVert\left\lVert\nabla(\eta^s-\eta^0)\right\rVert+k^2\left\lVert\frac{J_s-1}{s}\right\rVert\left\lVert p^0\right\rVert\left\lVert\eta^s-\eta^0\right\rVert.
\end{eqnarray*}
We need to show that the limit of $R(s)$ exists and is zero. To do this, we first note that by the lemma \ref{DFdellemme}, the terms $(B(s)-I)/s$ and $(J(s)-1)/s$ are uniformly bounded. Thus it suffices only to show that $\eta^s\rightarrow\eta^0$ in $H^1_0(\Omega)$-strong and that the norm $H_0^1(\Omega)$ of $(\eta^s-\eta^0)/s$ is bounded. From (\ref{DFeqetas}) of $\eta^s$ and $\eta^0$, we see that for any $\phi'\in H^1_0(\Omega)$,
\begin{align*}
\int_{\Omega}B(s) \nabla\eta^s\nabla\phi' - k^2\eta^s\phi' J_s- f\circ T_s\phi' J_s dx=\int_{\Omega}\nabla\eta^0\nabla\phi' - k^2\eta^0\phi' - f\phi' dx
\end{align*}
\begin{align*}
\int_{\Omega} \nabla\eta^s\nabla\phi'dx+\int_{\Omega}[B(s)-I] \nabla\eta^s\nabla\phi' - k^2\eta^s\phi' J_s- f\circ T_s\phi' J_s dx=\int_{\Omega}\nabla\eta^0\nabla\phi' - k^2\eta^0\phi' - f\phi' dx
\end{align*}
\begin{align*}
\int_{\Omega} \nabla(\eta^s-\eta^0)\nabla\phi' - k^2(\eta^s J_s-\eta^0)\phi' dx=-\int_{\Omega}[B(s)-I] \nabla\eta^s\nabla\phi'dx+\int_{\Omega}
(f\circ T_s J_s - f)\phi' dx
\end{align*}
\begin{align*}
\int_{\Omega} \nabla(\eta^s-\eta^0)\nabla\phi' - k^2(\eta^s J_s-\eta^0J_s+\eta^0J_s-\eta^0)\phi' dx=-\int_{\Omega}[B(s)-I] \nabla\eta^s\nabla\phi'dx+\int_{\Omega}
(f\circ T_s J_s - f)\phi' dx
\end{align*}

\begin{eqnarray*}
\int_{\Omega} \nabla(\eta^s-\eta^0)\nabla\phi' - k^2(\eta^s-\eta^0)J_s\phi' dx&=&-\int_{\Omega}[B(s)-I] \nabla\eta^s\nabla\phi'+
f\circ T_s(J_s-1)\phi' dx \\ [0.2cm] &\;&+\int_{\Omega}(f\circ T_s - f)\phi'+\eta^0(J_s-1)\phi' dx.
\end{eqnarray*}
Taking $\phi'=\eta^s-\eta^0$, we get

\begin{eqnarray*}
\lVert\nabla(\eta^s-\eta^0)\rVert^2 - k^2J_s\lVert\eta^s-\eta^0\rVert^2&=&-\int_{\Omega}[B(s)-I] \nabla\eta^s\nabla(\eta^s-\eta^0)+f\circ T_s(J_s-1)(\eta^s-\eta^0)dx\\ [0.2cm] &\;&+\int_{\Omega}
(f\circ T_s - f)(\eta^s-\eta^0)+\eta^0(J_s-1)(\eta^s-\eta^0) dx.
\end{eqnarray*}
According to the lemma \ref{DFdellemme}, we have $\alpha\leq J_s\leq\beta$. Using the Cauchy Schwarz inequality

\begin{eqnarray*}
\lVert\nabla(\eta^s-\eta^0)\rVert^2 - k^2\beta\lVert\eta^s-\eta^0\rVert^2&\leq&\lVert B(s)-I\rVert \lVert\nabla\eta^s\rVert \lVert\nabla(\eta^s-\eta^0)\rVert \\[0.2cm]&\,&+\lVert f\circ T_s\rVert \lVert J_s-1\rVert \lVert\eta^s-\eta^0\rVert\\[0.2cm] &\;&+
\lVert f\circ T_s - f \rVert\lVert\eta^s-\eta^0\rVert+\lVert\eta^0\rVert \lVert J_s-1 \rVert\lVert\eta^s-\eta^0\rVert.
\end{eqnarray*}
Due to the fact that $\Omega$ is a bounded open lipschitzian domain, there exists a positive constant $c(\Omega)>0$ such that $\lVert\eta^s-\eta^0\rVert \leq c(\Omega)\lVert\nabla(\eta^s-\eta^0)\rVert$. This will allow us to write

\begin{eqnarray*}
\left(1- k^2c(\Omega)\beta\right)\lVert\nabla(\eta^s-\eta^0)\rVert^2&\leq&\lVert B(s)-I\rVert \lVert\nabla\eta^s\rVert \lVert\nabla(\eta^s-\eta^0)\rVert \\[0.2cm]&\,&+\lVert f\circ T_s\rVert \lVert J_s-1\rVert \lVert\nabla(\eta^s-\eta^0)\rVert c(\Omega)\\[0.3cm] &\;&+
\lVert f\circ T_s - f \rVert\lVert\nabla(\eta^s-\eta^0)\rVert c(\Omega)\\[0.2cm]&\,&+\lVert\eta^0\rVert \lVert J_s-1 \rVert\lVert\nabla(\eta^s-\eta^0)\rVert c(\Omega).
\end{eqnarray*}
So we have
\begin{eqnarray*}
\left(1- k^2c(\Omega)\beta\right)\lVert\nabla(\eta^s-\eta^0)\rVert&\leq&\lVert B(s)-I\rVert \lVert\nabla\eta^s\rVert +\lVert f\circ T_s\rVert \lVert J_s-1\rVert c(\Omega)\\[0.3cm] &\;&+
\lVert f\circ T_s - f \rVert c(\Omega)+\lVert\eta^0\rVert \lVert J_s-1 \rVert c(\Omega).
\end{eqnarray*}
The right hand side of this inequality tends to zero when $s$ goes to zero. But this does not mean directly that $\lVert\nabla(\eta^s-\eta^0)\rVert$ goes to zero. In fact it will depend on the sign of $1- k^2c(\Omega)\beta$. If $1- k^2c(\Omega)\beta$ is positive, then $\lVert\nabla(\eta^s-\eta^0)\rVert$ goes to zero. On the other hand if $1- k^2c(\Omega)\beta$ is negative, we can always find a positive constant $M$ such that $M+1- k^2c(\Omega)\beta>0$ and such that we have
\begin{eqnarray*}
\left(M+1- k^2c(\Omega)\beta\right)\lVert\nabla(\eta^s-\eta^0)\rVert&\leq&\lVert B(s)-I\rVert \lVert\nabla\eta^s\rVert +\lVert f\circ T_s\rVert \lVert J_s-1\rVert c(\Omega)\\[0.3cm] &\;&+
\lVert f\circ T_s - f \rVert c(\Omega)+\lVert\eta^0\rVert \lVert J_s-1 \rVert c(\Omega).
\end{eqnarray*}
So we can conclude that $\eta^s\rightarrow\eta^0$ in $H^1_0(\Omega)$-strong. To show that $(\nabla(\eta^s-\eta^0))/s$ is bounded, we divide the equation by $s>0$
\begin{eqnarray*}
\left(1- k^2c(\Omega)\beta\right)\left\lVert\frac{\nabla(\eta^s-\eta^0)}{s}\right\rVert&\leq&\left\lVert \frac{B(s)-I}{s}\right\rVert \lVert\nabla\eta^s\rVert +\left\lVert \frac{J_s-1}{s}\right\rVert\lVert f\circ T_s\rVert c(\Omega)\\[0.3cm] &\;&+
\left\lVert \frac{f\circ T_s - f }{s}\right\rVert c(\Omega)+\left\lVert \frac{J_s-1}{s}\right\rVert\lVert\eta^0\rVert  c(\Omega).
\end{eqnarray*}
The term on the right is bounded, by playing with the sign of $1- k^2c(\Omega)\beta$, we show that $(\nabla(\eta^s-\eta^0))/s$ is bounded. This means that $(\eta^s-\eta^0)/s$ is bounded in $H^1_0(\Omega)$ and therefore $(\eta^s-\eta^0)/s$ is bounded in $L^2(\Omega)$.

Then $R(s)\rightarrow 0$ when $s\rightarrow0,$ i.e. the term $R(x^0,p^0)$ is zero and the shape derivative is given by $d_sL(0,\eta^0,p^0).$ 

\end{proof}

\noindent
The use of the averaged adjoint revealed the possible occurrence of an extra term and provided a simpler expression of the former hypothesis \textbf{(H3)}. It turns out that the extra term can also be obtained by using the standard adjoint at t = 0 significantly
simplifying the checking of that condition.
\begin{corollary}
Consider the Lagrangian functional
\begin{eqnarray*}
(s,x,y)\mapsto G(s,x,y): [0,\tau]\times X\times Y\rightarrow\mathbb{R},\;\;\tau>0
\end{eqnarray*}
where $X$ and $Y$ are vector spaces and the function $y\mapsto G(s,x,y)$ is affine. Assume that $\textbf{H(0)}$ and the following assumptions are satisfied:\\
$\textbf{H(1)}$ for all $s\in [0,\tau]$, $X(s)\neq\emptyset$,  $g(s)$ is finite, and for each $x\in X(0)$, $Y(0,x)\neq\emptyset$,\\
$\textbf{H(2)}$ for all $x\in X(0)$ and $p\in Y(0,x)$ $d_sG(0,x,p)$  exists,\\
$\textbf{H(3'')}$ there exist $x^0\in X(0)$ and $p^0\in Y(0,x^0)$ such that the following limit exists
\begin{eqnarray*}
R(x^0,p^0)=\lim_{s\searrow 0}\int_0^1 d_xG\left( s, x^0+\theta(x^s-x^0),p^0;\frac{x^s-x^0}{s}\right) d\theta.
\end{eqnarray*}
Then, $dg(0)$ exists and there exist $x^0\in X(0)$ and $p^0\in Y(0,x^0)$ such that $dg(0)=d_sG(0,x^0,p^0)+R(x^0,p^0)$.
\end{corollary}

\section{Topological derivative with geometrical perturbation of the domain }
Consider the functional defined by
\begin{eqnarray}\label{fye}
J(\Omega)=\int_\Omega\lvert\nabla\eta-A\rvert^2dx+\int_\Omega\lvert\eta-\eta_d\rvert^2dx
\end{eqnarray}
where $\Omega$ is a domain of $\mathbb{R}^N$, $N=2,3$ with a lipschitzian $\partial\Omega$ boundary and $\eta\in H^1_0(\Omega)$ is a solution of the variational equation of state
\begin{eqnarray}\label{fvdedepart}
\exists\;\eta\in H^1_0(\Omega)\;\text{such that} \int_\Omega\nabla\eta\nabla v-k^2\eta vdx=\int_\Omega fvdx,\;\text{for all}\;v\in H^1_0(\Omega).
\end{eqnarray} 
Let $x_0\in\Omega\subset\mathbb{R}^N$, $E$ be a set $d$-rectifiable with $0\leq d < N$ and $E_r$ the $r$-dilatation of $E$. We recall that 
\begin{eqnarray*}
d_E(x)=\inf_{x_0\in E}\lvert x-x_0\rvert\;\;\text{and}\;\;E_r=\{x\in\mathbb{R}^N:\;d_E(x)\leq r\}.
\end{eqnarray*}
We define the auxiliary variable $s$ as the volume of the $r$-dilatation $E_r$. The perturbed domains are this time defined by $r\rightarrow\Omega_r=\Omega\backslash E_r$.
\begin{small}
$s$ is the volume of $E_r$. It is also given by $s=\alpha_{N-d}r^{N-d}$, where $\alpha_{N-d}$ is the volume of the ball of radius 1 in $\mathbb{R}^{N-d}$. Therefore the perturbation takes the form $\textbf{1}_{\Omega_r}=\textbf{1}_{\Omega}-\textbf{1}_{E_r}$.
We define the following sets 
\begin{eqnarray*}
E_r^o=\{x\in\mathbb{R}^N:\;\;d_E(x)<r \}\;\;\text{and}\;\;\;E_r'=E_r^o\backslash E =\{x\in\mathbb{R}^N:\;\; 0<d_E(x)<r \}.
\end{eqnarray*}
\begin{definition}
Let $E$ be a closed subset of $\mathbb{R}^N$.\\
i) The set of points of $\mathbb{R}^N$ which have a unique projection on $E$.
\begin{eqnarray*}
Unp(E)=\{y\in\mathbb{R}^N:\;\exists\;\text{an unique}\;x\in E\;\text{such that}\;d_E(y)=\lVert y-x\rVert \}
\end{eqnarray*} 
where $d_E(y)$ is the distance function from a point $y$ to $E$.\\
ii) The reach of a point $x\in E$ and the reach of $E$ are defined
\begin{eqnarray*}
reach(E,x)=\sup\{r>0:\;B(x,r)\subset Unp(E) \},\;\;reach(E)=\inf_{x\in E}reach(E,x).
\end{eqnarray*}
We say that $E$ is a positive reach set if $reach(E)>0$.
\end{definition}
\end{small}
The definition of $Unp(E)$ obviously implies the existence of a projection on $E$, a function $p_E:Unp(E)\rightarrow E$ which associates $x\in Unp(E)$ the unique point $p_E(x)\in E$ such that $d_E(x)=\lVert x-p_E(x)\rVert$. For $0< R< reach(E)$, the projection $p_E$ is equal to the gradient of the following convex function
\begin{eqnarray*}
f_E(x)=\frac{1}{2}\left(\lVert x\rVert^2-d_E^2(x)\right),\;\;p_E(x)=\nabla f_E(x)=x-\frac{1}{2}d_E^2(x).
\end{eqnarray*} 
\textbf{Hypothesis 1}\\ Let $E$ be a closed connected subset of $\mathbb{R}^N$ and $R>0$ such that $E_{2R}=\{x\in\mathbb{R}^N;\;d_E(x)\leq 2R \} \subset\Omega$. Suppose that $E$ has a positive reach greater than $2R$ $(reach(E)> 2R)$ and that $0< H^d(E)<\infty$ for an integer $d$, $0\leq d< N$.\\
 We also need the following theorem which extends the Lebesgue differentiation theorem from $d=0$ to $0\leq d< N$.
   \begin{theorem}
   Let $E$ be a compact subset of $\mathbb{R}^N$ and $0\leq d< N$ an integer. Let $\alpha_k$ denote the volume of the unit ball in $\mathbb{R}^k$. Suppose that
   \begin{itemize}
   \item[(1)] $E$ is a $d$-rectifiable subset of $\mathbb{R}^N$ such that $\partial E=E$ and $0< H^d(E)<\infty$,
   \item[(2)] $E$ has a positive reach, i.e. there is $R>0$ such that $d^2_E\in C^{1,1}(E_R)$
   \item[(3)] and $f$  is continuous in $E_R$. Then
   \begin{eqnarray*}
   \lim_{r\searrow 0}\frac{1}{\alpha_{N-d}r^{N-d}}\int_{E_r}fd\xi=\lim_{r\searrow 0}\frac{1}{\alpha_{N-d}r^{N-d}}\int_{E_r}fop_E d\xi=\int_E fdH^d.
   \end{eqnarray*}
   \end{itemize}
   \end{theorem}

   \subsection{Topological derivative using the Lagrangian and minimax methods with perturbation on the source term }
We define the auxiliary variable $s$ as the volume of the $r$-dilation $E_r$. The perturbed domains are then defined by
\begin{eqnarray*}
s\rightarrow\Omega_s=\Omega\backslash E_r
\end{eqnarray*}
$s$ is the volume of $E_r$. It is also given by $s=\alpha_{N-d}r^{N-d}$, where $\alpha_{N-d}$ is the volume of the ball of radius 1 in $\mathbb{R}^{N-d}$. Therefore the perturbation takes the form $\mathds{1}_{\Omega_s}=\mathds{1}_{\Omega}-\mathds{1}_{E_r}$.
we introduce a topological perturbation of the second member $f$, by posing for example
    \begin{equation*}
    f^s=
    \begin{cases}
    f(x)\;\;if\;\;x\in\Omega\backslash E_r\\[0.2cm] \gamma f(x)\;\;if\;\;x\in E_r
    \end{cases}
    \end{equation*}
    $f$ is a function assumed to be continuous at the point $x_0$ and that $f(x)$ is constant in a neighbourhood of $x_0$.
     \begin{eqnarray*}
    \int_{\Omega}\nabla\eta^s\nabla vdx-\int_{\Omega}k^2\eta^s vdx&=&\int_{\Omega}f^s vdx\\&=&\int_{\Omega\backslash E_r}f^s vdx +\int_{E_r}f^s vdx  =\int_{\Omega\backslash E_r}f vdx+\gamma \int_{E_r}f vdx\\ &=&\int_{\Omega}f vdx-\int_{E_r}f vdx+\gamma \int_{E_r}f vdx\\&=&\int_{\Omega}f vdx-(1-\gamma)\int_{E_r}f vdx=\int_{\Omega}[f-(1-\gamma)\chi_{E_r} f]v dx.
    \end{eqnarray*}
   This leads to
    \begin{eqnarray}\label{DelHelper}
    \int_{\Omega}\nabla\eta^s\nabla vdx-\int_{\Omega}k^2\eta^s vdx&=&\int_{\Omega}[f-(1-\gamma)\chi_{E_r} f]v dx.
   \end{eqnarray}
We can therefore define the Lagrangian depending on $s$ as
\begin{small}
 \begin{eqnarray*}
 L(s,\varphi,\Phi)=\int_\Omega\lvert\nabla\varphi-A\rvert^2dx+\int_\Omega\lvert\varphi-\eta_d\rvert^2dx+\int_{\Omega}\nabla\varphi\nabla \Phi dx-\int_{\Omega}k^2\varphi\Phi dx-\int_{\Omega}[f-(1-\gamma)\chi_{E_r} f]\Phi dx.
 \end{eqnarray*}
 \end{small}
 The derivative of the Lagrangian depending on $s$ with respect to $\varphi$ is:
 \begin{eqnarray*}
 d_{\varphi} L(s,\varphi,\Phi;\varphi')=\int_\Omega 2(\nabla\varphi-A)\nabla\varphi'+2(\varphi-\eta_d)\varphi'+\int_{\Omega}\nabla\varphi'\nabla \Phi dx-\int_{\Omega}k^2\varphi'\Phi dx.
  \end{eqnarray*}
  The initial adjoint state $p^0$ is a solution of $d_{\varphi} L(s,\eta^0,p^0;\varphi')=0$ for all $\varphi'\in H^1_0(\Omega)$ with $\eta^0=\eta=\eta^s$ for $s=0$. So the variational formulation of the adjoint equation of state is given by: Find $p^0\in H^1_0(\Omega)$ such that
  \begin{eqnarray}\label{adjdel}
   \int_\Omega 2(\nabla \eta^0-A)\nabla\varphi'+2(\eta^0-\eta_d)\varphi'+\int_{\Omega}\nabla\varphi'\nabla p^0 dx-\int_{\Omega}k^2\varphi'p^0 dx=0.
    \end{eqnarray}
The derivative of the Lagrangian depending on $s$ with respect to $\Phi$ is:
   \begin{eqnarray*}
   d_{\Phi} L(s,\varphi,\Phi;\Phi')=\int_{\Omega}\nabla\varphi\nabla \Phi' dx-\int_{\Omega}k^2\varphi\Phi' dx-\int_{\Omega}[f-(1-\gamma)\chi_{E_r} f]\Phi' dx.
    \end{eqnarray*}
  The initial state $\eta=\eta^0$ is solution of $d_{\Phi} L(0,\eta^0,0;\Phi')=0$ for all $\Phi'\in H^1_0(\Omega)$ and we find obviously
 \begin{eqnarray*}
   \int_{\Omega}\nabla\eta^0\nabla \Phi' dx-\int_{\Omega}k^2\eta^0\Phi dx=\int_{\Omega}f\Phi' dx\;\;\text{pour tout}\; \Phi'\in H^1_0(\Omega)
 \end{eqnarray*}
 and the state $\eta^s$ for $s\geq 0$ verifies
  \begin{eqnarray}\label{etadel}
  \int_{\Omega}\nabla\eta^s\nabla \Phi' dx-\int_{\Omega}k^2\eta^s\Phi' dx=\int_{\Omega}[f-(1-\gamma)\chi_{E_r} f]\Phi' dx\;\;\text{pour tout}\; \Phi'\in H^1_0(\Omega).
  \end{eqnarray}
  Now we want to calculate the derivative of the Lagrangian depending on $s$ with respect to $s$. To do this we first calculate the ratio $\frac{L(s,\varphi,\Phi)-L(0,\varphi,\Phi)}{s}$
   \begin{eqnarray*}
   \frac{L(s,\varphi,\Phi)-L(0,\varphi,\Phi)}{s}&=&\frac{1-\gamma}{s}\int_{E_r}f(x)\Phi(x)dx\\&=&\frac{1-\gamma}{\vert E_r\rvert}\int_{E_r}f(x)\Phi(x)dx.
   \end{eqnarray*} 
   For $d=0$, $E=\{x_0\}$, $E_r=\{x\in\mathbb{R}^N:\;\lvert x-x_0\rvert\leq r\}=\overline{B(x_0,r)}$; the closed ball of centre $x_0$ and radius $r$. So by applying Lebesgue's differentiation theorem we get $d_sL(0,\varphi,\Phi)=(1-\gamma)f(x_0)\Phi(x_0)$. So at the point $\eta^0$ and $p^0$
     \begin{eqnarray}\label{delfourdt}
     d_sL(0,\eta^0,p^0)=(1-\gamma)f(x_0)p^0(x_0).
     \end{eqnarray}
In cases where $d=1,2$, we have
 \begin{eqnarray*}
  \frac{L(s,\varphi,\Phi)-L(0,\varphi,\Phi)}{s}&=&\frac{1-\gamma}{\vert E_r\rvert}\int_{E_r}f(x)\Phi(x)dx\\&=&
  \frac{1-\gamma}{\alpha_{N-d}r^{N-d}}\int_{E_r}f(x)\Phi(x)dx\rightarrow (1-\gamma)\int_{E}f(x)\Phi(x)dH^d.
  \end{eqnarray*} 
Therfore at the point $\eta^0$ et $p^0$  
  \begin{eqnarray}\label{delfourdt1}
     d_sL(0,\eta^0,p^0)=(1-\gamma)\int_{E}f(x)\Phi(x)dH^d.
     \end{eqnarray}
We now need to calculate $R(\eta^0,p^0)$ which is defined by
\begin{eqnarray}
R(\eta^0,p^0)=\lim_{s\searrow 0}\int_0^1 d_xL\left(s,\eta^0+\theta(\eta^s-\eta^0),p^0;\frac{\eta^s-\eta^0}{s}\right)=\lim_{s\searrow 0}\;R(s)
\end{eqnarray} 
\begin{eqnarray*}
 R(s)&=&\int_\Omega 2\left\lbrace \nabla\left(\frac{\eta^s+\eta^0}{2}\right)-A\right\rbrace\nabla\left(\frac{\eta^s-\eta^0}{s}\right)
 +2\left(\frac{\eta^s+\eta^0}{2}-\eta_d\right)\left(\frac{\eta^s-\eta^0}{s}\right)dx\\ [0.2cm]&+&\int_\Omega\nabla\left(\frac{\eta^s-\eta^0}{s}\right)\nabla p^0-k^2\left(\frac{\eta^s-\eta^0}{s}\right)p^0dx.
 \end{eqnarray*} 
 Taking $\varphi'=\frac{\eta^s-\eta^0}{s}$ in the adjoint equation of state verified by $p^0$ (\ref{adjdel}), we obtain
 \begin{eqnarray*}
&\;& \int_\Omega 2(\nabla \eta^0-A)\nabla\left(\frac{\eta^s-\eta^0}{s}\right)+2(\eta^0-\eta_d)\left(\frac{\eta^s-\eta^0}{s}\right)dx+\\&\;&\int_{\Omega}\nabla\left(\frac{\eta^s-\eta^0}{s}\right)\nabla p^0 dx-\int_{\Omega}k^2\left(\frac{\eta^s-\eta^0}{s}\right)p^0 dx=0.
   \end{eqnarray*}
So injecting this equality into the expression of $R(s)$ gives
  \begin{eqnarray*}
   R(s)&=&\int_\Omega\nabla\left(\eta^s-\eta^0\right)\nabla\left(\frac{\eta^s-\eta^0}{s}\right)
   +\left(\eta^s-\eta^0\right)\left(\frac{\eta^s-\eta^0}{s}\right)\\&=&\int_\Omega\left\lvert \nabla\left(\frac{\eta^s-\eta^0}{s^{1/2}}\right)\right\rvert^2dx+\int_\Omega\left\lvert \frac{\eta^s-\eta^0}{s^{1/2}}\right\rvert^2dx =\left\lVert \frac{\eta^s-\eta^0}{s^{1/2}}\right\rVert^2_{H^1(\Omega)} 
   \end{eqnarray*} 
To show that $R(s)$ tends to $0$, we can show that the norm of $\eta^s-\eta^0$ in $H^1(\Omega)$ is an $O(s)$. In this case we need the following lemma
\begin{lemma}
Let $\eta^0$ and $\eta^s$ be the solutions of  (\ref{fvdedepart}) and (\ref{DelHelper}) respectively. Then, we have the following estimate:
 \begin{eqnarray*}
    \lVert \eta^s-\eta^0 \rVert_{H^1(\Omega)}\leq Cs
 \end{eqnarray*}
    with $C$ a positive constant not depending of $s.$
\end{lemma}
\begin{proof}
 Let $a(\eta,v)$ be the bilinear formula defined by
      \begin{eqnarray*}
      a(\eta,v)=\int_\Omega\nabla\eta \nabla vdx-k^2\int_\Omega\eta vdx
      \end{eqnarray*}
       it is clear that there are two constants $\alpha\in\mathbb{R}$ and $\beta >0$ such that 
            \begin{eqnarray}
            \lvert a(\eta,\eta)\rvert\geq \beta\lVert\eta\rVert^2_{H^1_0(\Omega)}-\alpha \lVert \eta\rVert^2_{L^2(\Omega)}
            \end{eqnarray}
   indeed, we have
         \begin{eqnarray*}
          a(\eta,v)&=&\int_{\Omega}\lvert\nabla\eta\rvert^2 dx-\int_{\Omega}k^2\lvert\eta\rvert^2 dx=\lVert\nabla\eta\rVert^2_{L^2(\Omega)}-k^2 \lVert \eta\rVert^2_{L^2(\Omega)}\\&=&\lVert\eta\rVert^2_{H^1_0(\Omega)}-k^2 \lVert \eta\rVert^2_{L^2(\Omega)}\\&\geq&\beta\lVert\eta\rVert^2_{H^1_0(\Omega)}-k^2 \lVert \eta\rVert^2_{L^2(\Omega)}.
          \end{eqnarray*}
    Therefore $a(\eta,v)$ is $H^1_0(\Omega)$ elliptical for $\alpha=k^2$. Subtracting the two problems:
        \begin{eqnarray*}
        \int_{\Omega}\nabla(\eta^s-\eta^0)\nabla vdx-\int_{\Omega}k^2(\eta^s-\eta^0) vdx&=&\int_{\Omega}f^s vdx-\int_{\Omega}f vdx \\ &=&\int_{\Omega\backslash E_r}f vdx+\gamma \int_{E_r}f vdx-\int_{\Omega}fvdx\\&=&\int_{\Omega}f vdx-(1-\gamma)\int_{E_r}f vdx-\int_{\Omega}f v\\&=&-(1-\gamma)\int_{E_r}f vdx.
        \end{eqnarray*}
        This leads to
        \begin{eqnarray}
        \int_{\Omega}\nabla(\eta^s-\eta^0)\nabla vdx-\int_{\Omega}k^2(\eta^s-\eta^0) vdx=-(1-\gamma)\int_{E_r}f vdx.
        \end{eqnarray}
        Since $\eta^s-\eta^0\in H^1_0(\Omega)$, we can take $v=\eta^s-\eta^0$. This gives
        \begin{eqnarray*}
        \int_{\Omega}\lvert\nabla(\eta^s-\eta^0)\rvert^2dx-\int_{\Omega}k^2(\eta^s-\eta^0)\rvert^2 dx=-(1-\gamma)\int_{E_r}f(\eta^s-\eta^0) dx.
        \end{eqnarray*}
        Using the inequality verified by $a(\eta,\eta)$
        \begin{eqnarray*}
           \beta\lVert\eta^s-\eta^0\rVert^2_{H^1_0(\Omega)}-k^2 \lVert\eta^s-\eta^0\rVert^2_{L^2(\Omega)}\leq\lvert a(\eta^s-\eta^0,\eta^s-\eta^0)\rvert\leq \lvert 1-\gamma\rvert\left\lvert\int_{E_r} f(\eta^s-\eta^0) dx\right\rvert.
            \end{eqnarray*}
        The Cauchy Schwarz inequality leads to:
        \begin{eqnarray*}
        \beta\lVert\eta^s-\eta^0\rVert^2_{H^1_0(\Omega)}-k^2 \lVert\eta^s-\eta^0\rVert^2_{L^2(\Omega)} &\leq&  \lvert 1-\gamma\rvert\lVert f\rVert_{L^2(E_r)}\lVert \eta^s-\eta^0\rVert_{L^2(E_r)}\\ &\leq& C_1s\lVert \eta^s-\eta^0\rVert_{L^2(E_r)}\\ &\leq& C_2s \lVert\eta^s-\eta^0\rVert_{H^1(\Omega)}.
        \end{eqnarray*}
        We can also write 
         \begin{eqnarray*}
          \beta\lVert\eta^s-\eta^0\rVert^2_{H^1_0(\Omega)}-k^2 \lVert\eta^s-\eta^0\rVert^2_{H^1(\Omega)} \leq C_2s \lVert\eta^s-\eta^0\rVert_{H^1(\Omega)}.
          \end{eqnarray*}
         As the norms $\lVert\cdot\rVert_{H^1_0(\Omega)}$ and $\lVert\cdot\rVert_{H^1(\Omega)}$ are equivalent, we can find $\beta_1>0$ such that
          \begin{eqnarray*}
             \beta_1\lVert\eta^s-\eta^0\rVert^2_{H^1(\Omega)}\leq \lVert\eta^s-\eta^0\rVert^2_{H^1_0(\Omega)}.
             \end{eqnarray*}
             Therefore, taking $\beta_2$ as the product of $\beta$ and $\beta_1$
             \begin{eqnarray*}
                 \beta_2\lVert\eta^s-\eta^0\rVert^2_{H^1(\Omega)}-k^2 \lVert \eta^s-\eta^0\rVert^2_{H^1(\Omega)} \leq C_2s \lVert\eta^s-\eta^0\rVert_{H^1(\Omega)}.
                 \end{eqnarray*}
   \begin{eqnarray*}
    (\beta_2-k^2)\lVert\eta^s-\eta^0\rVert_{H^1(\Omega)}\leq C_2s.
   \end{eqnarray*}               
 If $\beta_2-k^2>0$, then we have, for $C=\frac{C_2}{\beta_2-k^2}$
         \begin{eqnarray*}
                   \lVert\eta^s-\eta^0\rVert^2_{H^1(\Omega)} \leq Cs.
         \end{eqnarray*}
  If $\beta_2-k^2 \leq 0$, we can find $M_1>0$ such that $M_1+\beta_2-k^2> 0$ and such that
         \begin{eqnarray*}
           (M_1+\beta_2-k^2)\lVert \eta^s-\eta^0\rVert_{H^1(\Omega)}\leq C_2s
         \end{eqnarray*}
           and so taking $C=\frac{C_2}{M_1+\beta_2-k^2}$, the desired result is obtained.
\end{proof}
 Because $\lVert\eta^s-\eta^0\rVert=O(s)$, by definition, they exist a rank $N\in\mathbb{N}$ and a bounded sequence $b_s$ such that for all $s\geq N$, $\lVert\eta^s-\eta^0\rVert_{H^1(\Omega)}=sb_s$. This implies that
   \begin{align}
   \left\lVert\frac{\eta^s-\eta^0}{s^{1/2}}\right\rVert^2_{H^1(\Omega)}=s(b_s)^2.
   \end{align}
   Thus by tending $s$ to zero, $s(b_s)^2$ goes to zero. We conclude that $\lim_{s\searrow 0} R(s)=0$. We then have the following proposition
\begin{prop1}
For a compact $d$-rectifiable set $E\subset \Omega$ ($d=0$) the topological derivative for the functionnal $J$ is
\begin{eqnarray*}
D_TJ(x_0)=(1-\gamma)f(x_0)p^0(x_0).
   \end{eqnarray*}
For a compact $d$-rectifiable set $E\subset \Omega$ ($d=1$ for a curve, $d=2$ for a surface) the topological derivative for the functionnal $J$ is
\begin{eqnarray*}
     D_TJ=(1-\gamma)\int_{E}f(x)\Phi(x)dH^d
     \end{eqnarray*}
\end{prop1}

\subsection{Dirichlet condition on the boundary of the hole}
Associate with $r$, $0< r\leq R$ the perturbed domain $\Omega_r=\Omega\backslash E_r$, where by assumption, $\partial\Omega_r=\partial\Omega\cup\partial E_r$, $\partial\Omega\cap\partial E_r=\emptyset$ and $\partial E_r\in C^{1,1}$.
We first want to put a Dirichlet condition at the boundary of the hole (the $r$-dilatation). Let $\eta_r$ be the solution of the following system
\begin{equation}\label{Dirdelf}
\begin{cases}
-\Delta\eta_r-k^2\eta_r=f\qquad\text{in}\;\;\Omega_r\\[0.2cm]
\qquad\qquad\;\;\, \eta_r=0\qquad\text{on}\;\;\partial\Omega\\[0.2cm]
\qquad\qquad\;\;\, \eta_r=0\qquad\text{on}\;\;\partial E_r
\end{cases}
\end{equation}
which can be extended to $\Omega$ by introducing the solution $\eta^0_r:E_r^0 \rightarrow\mathbb{R}$ of the problem
\begin{eqnarray*}
-\Delta\eta^0_r-k^2\eta^0_r=f\;\;\text{in}\;\;E^0_r,\;\;\eta^0_r=\eta_r\;\text{on}\;\partial E_r.
\end{eqnarray*}
We suppose that $\Omega_r$ has two components $\Omega_r^m$ and $\Omega_r^0$. $\Omega_r^m$ is the component of $\Omega_r$ for which $\partial\Omega$ is part of its boundary. $\Omega_r^0$ is the blind component of $\Omega_r$ whose boundary has an empty intersection with $\partial\Omega$. The function $\eta_r$ is distributed between the two components $\Omega_r^0$ and $\Omega_r^m$ as
\begin{eqnarray}
\eta_r=\eta_0 \;\;\text{in}\;\;\Omega_r^0\;\;\;\text{and}\;\;\; -\Delta\eta_r-k^2\eta_r=f \;\;\text{in}\;\;  \Omega_r^m,
\end{eqnarray}
\begin{eqnarray*}
\begin{cases}
\eta_r=0\;\;\text{on}\;\;\partial\Omega\\[0.2cm] \eta_r=0\;\;\text{on}\;\;\partial\Omega_r^m\cap\partial E_r.
\end{cases}
\end{eqnarray*}
Since $\partial E_r$ is made up of two disjoint boundary $\partial\Omega_r^0$ and $ \partial\Omega_r^m\cap\partial E_r$, we can construct an extension to $\Omega$ by defining the solution $ \eta^0_r: E_r^0\rightarrow\mathbb{R}$ of the problem
\begin{eqnarray*}
-\Delta\eta^0_r-k^2\eta^0_r=f \;\;\text{in}\;\;  E_r^0,\;\;\;\;\eta_r^0&=&\eta_r\;\;\text{on}\;\;\partial\Omega_r^m\cap\partial E_r\\\eta_r^0&=&\eta_r\;\;\text{on}\;\;\partial\Omega_r^0.
\end{eqnarray*}
For simplicity, we can assume that $\Omega_r^0$ is empty.\\
Let us associate with the state $\eta_r$ the perturbed objective function
\begin{eqnarray*}
J(\Omega_r)=\int_{\Omega_r}\left(\lvert \nabla\eta_r-A\rvert^2+\lvert\eta_r-\eta_d\lvert^2\right) dx.
\end{eqnarray*}
We want to find the topological derivative of the objective function $J(\Omega_r)$
\begin{eqnarray}\label{delderiveD}
dJ=\lim_{r\searrow 0}\frac{J(\Omega_r)-J(\Omega)}{\alpha_{N-d}r^{N-d}}.
\end{eqnarray}
We define the following set
\begin{eqnarray}
\mathcal{V}_r=\left\{\eta\in H^1(\Omega_r):\;\;\eta=0\;\text{on}\;\partial\Omega,\;\eta=0\;\text{on}\;\partial E_r \right\}.
\end{eqnarray}
The variational formulation of (\ref{Dirdelf}) is to find $\eta_r\in \mathcal{V}_r$ such that
\begin{eqnarray*}
-\int_{\partial\Omega_r}\frac{\partial\eta_r}{\partial n}vd\sigma+\int_{\Omega_r}\nabla\eta_r\nabla vdx-\int_{\Omega_r}k^2\eta_r vdx=\int_{\Omega_r} fvdx,\;\;\text{for all}\;v\in\mathcal{V}
\end{eqnarray*}
that is to say
\begin{eqnarray*}
-\int_{\partial\Omega}\frac{\partial\eta_r}{\partial n}vd\sigma-\int_{\partial E_r}\frac{\partial\eta_r}{\partial n}vd\sigma+\int_{\Omega_r}\nabla\eta_r\nabla vdx-\int_{\Omega_r}k^2\eta_r vdx=\int_{\Omega_r} fvdx,\;\;\text{for all}\;v\in\mathcal{V}.
\end{eqnarray*}
By taking $\mathcal{V}=\mathcal{V}_r$, we get
\begin{eqnarray}\label{pbavecper}
\int_{\Omega_r}\nabla\eta_r\nabla vdx-\int_{\Omega_r}k^2\eta_r vdx=\int_{\Omega_r} fvdx,\;\;\text{for all}\;v\in \mathcal{V}_r.
\end{eqnarray}
We can therefore define the $r$-dependent Lagrangian as
 \begin{eqnarray*}
 L(r,\varphi,\phi)=\int_{\Omega_r}\lvert\nabla\varphi-A\rvert^2dx+\int_{\Omega_r}\lvert\varphi-\eta_d\rvert^2 dx+\int_{\Omega_r}\nabla\varphi\nabla \phi dx-\int_{\Omega_r}k^2\varphi\phi dx-\int_{\Omega_r} f\phi dx.
 \end{eqnarray*}
 \begin{eqnarray*}
 J(\Omega_r)=\inf_{\varphi\in\mathcal{V}_r}\sup_{\phi\in\mathcal{V}_r} L(r,\varphi,\phi)=\int_{\Omega_r}\left\lbrace\lvert\nabla\eta_r-A\rvert^2+\lvert\eta_r-\eta_d\rvert^2+\nabla\eta_r\nabla p^0-k^2\eta_r p_0-fp_0\right\rbrace dx.
 \end{eqnarray*}
 \begin{eqnarray*}
  J(\Omega)=\inf_{\varphi\in H^1_0(\Omega)}\sup_{\phi\in H^1_0(\Omega)} L(0,\varphi,\phi)=\int_{\Omega}\left\lbrace\lvert\nabla\eta_0-A\rvert^2+\lvert\eta_0-\eta_d\rvert^2+\nabla\eta_0\nabla p_0-k^2\eta_0 p_0-fp_0\right\rbrace dx.
  \end{eqnarray*}
We can rewrite the $r$-dependent Lagrangian as follows
 \begin{eqnarray*}
 L(r,\varphi,\phi)&=&\int_{\Omega}\left\lbrace\lvert\nabla\varphi-A\rvert^2+\lvert\varphi-\eta_d\rvert^2+\nabla\varphi\nabla \phi -k^2\varphi\phi-f\phi\right\rbrace dx\\&-&\int_{E_r}\left\lbrace\lvert\nabla\varphi-A\rvert^2+\lvert\varphi-\eta_d\rvert^2+\nabla\varphi\nabla \phi -k^2\varphi\phi-f\phi\right\rbrace dx.
 \end{eqnarray*}
The derivative of the $r$-dependent Lagrangian with respect to $\varphi$ is:
 \begin{eqnarray*}
 d_{\varphi} L(r,\varphi,\phi;\varphi')=\int_{\Omega_r}\left\lbrace 2(\nabla\varphi-A)\nabla\varphi'+2(\varphi-\eta_d)\varphi'\right\rbrace dx+\int_{\Omega_r}\nabla\varphi'\nabla \phi dx-\int_{\Omega_r}k^2\varphi'\phi dx.
  \end{eqnarray*}
The initial adjoint state $p_0$ is a solution of $d_{\varphi} L(0,\eta_0,p_0;\varphi')=0$ for all $\varphi'\in H^1_0(\Omega)$ with $\eta_0=\eta=\eta_r$ for $r=0$. Thus the variational formulation of the adjoint equation of state is given by

\begin{eqnarray}\label{etatp0D}
\text{Find}\,p_0\in H^1_0(\Omega),\;\, \int_\Omega \left\lbrace2(\nabla \eta_0-A)\nabla\varphi'+2(\eta_0-\eta_d)\varphi'\right\rbrace dx+\int_{\Omega}\left\lbrace\nabla\varphi'\nabla p_0 -k^2\varphi'p_0\right\rbrace dx=0.
  \end{eqnarray}
The derivative of the Lagrangian depending on $r$ with respect to $\phi$ is:
   \begin{eqnarray*}
   d_{\phi} L(r,\varphi,\phi;\phi')=\int_{\Omega_r}\nabla\varphi\nabla \phi' dx-\int_{\Omega_r}k^2\varphi\phi' dx-\int_{\Omega_r} f\phi' dx.
    \end{eqnarray*}
  The initial state $\eta=\eta_0$ is solution of $d_{\phi} L(0,\eta_0,0;\phi')=0$ for all $\phi'\in H^1_0(\Omega)$ and we obviously find the initial variational equation
 \begin{eqnarray}\label{etateta0D}
   \int_{\Omega}\nabla\eta_0\nabla \phi' dx-\int_{\Omega}k^2\eta_0\phi' dx=\int_{\Omega} f\phi' dx,\;\;\text{for all}\; \phi'\in H^1_0(\Omega)
 \end{eqnarray}
and the state $\eta_r$ for $r\geq 0$ verifies
  \begin{eqnarray}
  \int_{\Omega_r}\nabla\eta_r\nabla \phi' dx-\int_{\Omega_r}k^2\eta_r\phi' dx=\int_{\Omega_r} f\phi dx,\;\;\text{for all}\; \phi'\in H^1_0(\Omega).
  \end{eqnarray}
 Now we want to calculate the derivative of the Lagrangian depending on $r$ with respect to $s$. To do this we first calculate the ratio $\frac{L(r,\varphi,\phi)-L(0,\varphi,\phi)}{s}$
 \begin{eqnarray*}
 L(r,\varphi,\phi)-L(0,\varphi,\phi)&=&\int_{\Omega_r}\left\lbrace\lvert\nabla\varphi-A\rvert^2+\lvert\varphi-\eta_d\rvert^2\right\rbrace dx+\int_{\Omega_r}\left\lbrace\nabla\varphi\nabla \phi -k^2\varphi\phi-f\phi\right\rbrace dx \\&-&\int_{\Omega}\left\lbrace\lvert\nabla\varphi-A\rvert^2+\lvert\varphi-\eta_d\rvert^2\right\rbrace dx+\int_{\Omega}\left\lbrace\nabla\varphi\nabla \phi -k^2\varphi\phi-f\phi\right\rbrace dx \\&=&-\int_{E_r}\left\lbrace\lvert\nabla\varphi-A\rvert^2+\lvert\varphi-\eta_d\rvert^2+\nabla\varphi\nabla \phi -k^2\varphi\phi-f\phi\right\rbrace dx.
 \end{eqnarray*} 
For $d=0$, $E=\{x_0\}$, $E_r=\{x\in\mathbb{R}^N:\;\lvert x-x_0\rvert\leq r\}=\overline{B(x_0,r)}$; the closed ball centred $x_0$ and radius $r$. So by applying Lebesgue's differentiation theorem we get
 \begin{eqnarray*}
    d_sL(0,\varphi,\phi)&=&\lim_{s\searrow 0}\;\; -\frac{1}{\lvert B(x_0,r)\rvert}\int_{B(x_0,r)}\left\lbrace\lvert\nabla\varphi-A\rvert^2+\lvert\varphi-\eta_d\rvert^2+\nabla\varphi\nabla \phi -k^2\varphi\phi-f\phi\right\rbrace dx\\[0.2cm]
    &=&-\lvert\nabla\varphi(x_0)-A(x_0)\rvert^2-\lvert\varphi(x_0)-\eta_d(x_0)\rvert^2-\nabla\varphi(x_0)\nabla \phi(x_0) \\[0.2cm] &+&k^2\varphi(x_0)\phi(x_0)+f(x_0)\phi(x_0).
  \end{eqnarray*}
So at the point $\eta_0$ and $p_0$
   \begin{eqnarray*}
   d_sL(0,\eta_0,p_0)&=&-\lvert\nabla\eta_0(x_0)-A(x_0)\rvert^2-\lvert\eta_0(x_0)-\eta_d(x_0)\rvert^2-\nabla\eta_0(x_0)\nabla p_0(x_0) \\[0.2cm] &+&k^2\eta_0(x_0)p_0(x_0)+f(x_0)p_0(x_0).
   \end{eqnarray*}
In cases where $0<d\leq N-1$, we have
 \begin{eqnarray*}
  \frac{L(r,\varphi,\Phi)-L(0,\varphi,\Phi)}{s}&=&-\frac{1}{\vert E_r\rvert}\int_{E_r}\left(\lvert\nabla\varphi-A\rvert^2+\lvert\varphi-\eta_d\rvert^2+\nabla\varphi\nabla \phi -k^2\varphi\phi-f\phi\right) dx\\&=&
  -\frac{1}{\alpha_{N-d}r^{N-d}}\int_{E_r}\left(\lvert\nabla\varphi-A\rvert^2+\lvert\varphi-\eta_d\rvert^2+\nabla\varphi\nabla \phi -k^2\varphi\phi -f\phi\right) dx\\&\rightarrow& -\int_{E}\left(\lvert\nabla\varphi-A\rvert^2+\lvert\varphi-\eta_d\rvert^2+\nabla\varphi\nabla \phi -k^2\varphi\phi-f\phi\right) dH^d.
  \end{eqnarray*} 
In the cases $\eta_0$ and $p_0$  
  \begin{eqnarray*}
     d_sL(0,\eta_0,p_0)=-\int_{E}\left(\lvert\nabla\eta_0-A\rvert^2+\lvert\eta_0-\eta_d\rvert^2+\nabla\eta_0\nabla p_0 -k^2\eta_0p_0-fp_0\right) dH^d.
 \end{eqnarray*}
We define the function $R(r)$ by
\begin{eqnarray}
R(r)=\int_0^1 d_xL\left(r,\eta_0+\theta(\eta_r-\eta_0),p_0;\frac{\eta_r-\eta_0}{s}\right).
\end{eqnarray} 
 \begin{eqnarray*}
 R(r)&=&\int_{\Omega_r} 2\left\lbrace \nabla\left(\frac{\eta_r+\eta_0}{2}\right)-A\right\rbrace\nabla\left(\frac{\eta_r-\eta_0}{s}\right)
 +2\left(\frac{\eta_r+\eta_0}{2}-\eta_d\right)\left(\frac{\eta_r-\eta_0}{s}\right)dx\\ [0.2cm]&+&\int_{\Omega_r}\nabla\left(\frac{\eta_r-\eta_0}{s}\right)\nabla p_0-k^2\left(\frac{\eta_r-\eta_0}{s}\right)p_0dx\\
 &=&\frac{1}{s}\int_{\Omega_r}2\left\lbrace \nabla\left(\frac{\eta_r+\eta_0}{2}\right)-A\right\rbrace\nabla\left(\eta_r-\eta_0\right)
  +2\left(\frac{\eta_r+\eta_0}{2}-\eta_d\right)\left(\eta_r-\eta_0\right)dx\\ [0.2cm]&+&\frac{1}{s}\int_{\Omega_r}\nabla\left(\eta_r-\eta_0\right)\nabla p_0-k^2\left(\eta_r-\eta_0\right)p_0dx\\
  \end{eqnarray*} 
   \begin{eqnarray*}
   &=&\frac{1}{s}\int_{\Omega_r}2\left\lbrace \nabla\left(\frac{\eta_r+\eta_0}{2}\right)-\nabla\eta_0+\nabla\eta_0-A\right\rbrace\nabla\left(\eta_r-\eta_0\right)dx\\ [0.2cm]&+&\frac{1}{s}\int_{\Omega_r}
    2\left(\frac{\eta_r+\eta_0}{2}-\eta_0+\eta_0-\eta_d\right)\left(\eta_r-\eta_0\right)+\nabla\left(\eta_r-\eta_0\right)\nabla p_0-k^2\left(\eta_r-\eta_0\right)p_0dx
    \\
       &=&\frac{1}{s}\int_{\Omega_r}2\left\lbrace \nabla\left(\frac{\eta_r-\eta_0}{2}\right)+\nabla\eta_0-A\right\rbrace\nabla\left(\eta_r-\eta_0\right)
        +2\left(\frac{\eta_r-\eta_0}{2}+\eta_0-\eta_d\right)\left(\eta_r-\eta_0\right)dx\\ [0.2cm]&+&\frac{1}{s}\int_{\Omega_r}\nabla\left(\eta_r-\eta_0\right)\nabla p_0-k^2\left(\eta_r-\eta_0\right)p_0dx.
 \end{eqnarray*} 
Then we get
\begin{equation}\label{R(s)2}
R(r)=\int_{\Omega_r}\left\lvert \nabla\left(\frac{\eta_r-\eta_0}{s^{1/2}}\right)\right\rvert^2dx+\int_{\Omega_r}\left\lvert\frac{\eta_r-\eta_0}{s^{1/2}}\right\rvert^2dx+\frac{1}{s}\int_{\Omega_r}\nabla\left(\eta_r-\eta_0\right)\nabla p_0-k^2\left(\eta_r-\eta_0\right)p_0dx\\ [0.2cm]$$
$$\frac{1}{s}\int_{\Omega_r}2(\nabla\eta_0-A)\nabla(\eta_r-\eta_0)+2(\eta_0-\eta_d)(\eta_r-\eta_0)dx.
\end{equation}
An other writing of (\ref{fvdedepart}) shows that for $\eta\in H_0^1(\Omega)$,
\begin{eqnarray*}
\int_{\Omega_r}\left(\nabla\eta\nabla v-k^2\eta vdx-fv\right) dx=-\int_{E_r}\left(\nabla\eta\nabla v-k^2\eta vdx-fv\right) dx=\int_{\partial E_r}\frac{\partial\eta}{\partial n}vdH^{N-1},\;\;\forall\;v\in H_0^1(\Omega).
\end{eqnarray*}
Thanks to the assumption on $E$, $\partial E_r\in C^{1,1}$
\begin{eqnarray*}
\frac{\partial\eta}{\partial n}=\nabla\eta\cdot n_{\Omega_r}=-\nabla\eta\cdot n_{\partial E_r}=-\nabla\eta\cdot\nabla_{d_E}\;\text{sur}\;E_r.
\end{eqnarray*}
This allows us to write that
\begin{eqnarray}\label{initialr}
\int_{\Omega_r}\left(\nabla\eta\nabla v-k^2\eta vdx-fv\right) dx=-\int_{\partial E_r}\nabla\eta\cdot\nabla_{d_E}v dH^{N-1}.
\end{eqnarray}
Between (\ref{pbavecper}) and (\ref{initialr}), we obtain 
\begin{eqnarray*}
\int_{\Omega_r}\left\lbrace\nabla(\eta_r-\eta_0) \nabla v-k^2(\eta_r-\eta_0) v\right\rbrace dx=\int_{\partial E_r}\nabla\eta_0\cdot\nabla_{d_E}v dH^{N-1}.
\end{eqnarray*}
At this level, we introduce the adjoint equation for $r\geq 0$
\begin{eqnarray}\label{deladDir}
\int_{\Omega_r}\left\lbrace 2(\nabla \eta_0-A)\nabla\varphi'+2(\eta_0-\eta_d)\varphi'\right\rbrace dx+\int_{\Omega_r}\nabla\varphi'\nabla p_r dx-\int_{\Omega_r}k^2\varphi'p_r dx=0.
  \end{eqnarray}
 By taking $\varphi'=\eta_r-\eta_0$ in (\ref{deladDir}), it comes
  \begin{eqnarray*}
  \int_{\Omega_r}\left\lbrace 2(\nabla \eta_0-A)\nabla(\eta_r-\eta_0)+2(\eta_0-\eta_d)(\eta_r-\eta_0)\right\rbrace dx=-\int_{\Omega_r}\left\lbrace\nabla(\eta_r-\eta_0)\nabla p_r -k^2(\eta_r-\eta_0)p_r\right\rbrace dx.
    \end{eqnarray*}
We can therefore rewrite it as follows
\begin{eqnarray}
R(r)&=&\int_{\Omega_r}\left\lvert \nabla\left(\frac{\eta_r-\eta_0}{s^{1/2}}\right)\right\rvert^2dx+\int_{\Omega_r}\left\lvert \frac{\eta_s-\eta_0}{s^{1/2}}\right\rvert^2dx+\frac{1}{s}\int_{\Omega_r}\left\lbrace\nabla\left(\eta_r-\eta_0\right)\nabla p_0-k^2\left(\eta_r-\eta_0\right)p_0\right\rbrace dx\\ [0.2cm]&-&
\frac{1}{s}\int_{\Omega_r}\left\lbrace\nabla(\eta_r-\eta_0)\nabla p_r -k^2(\eta_r-\eta_0)p_r\right\rbrace dx\\
&=&\int_{\Omega_r}\left\lvert \nabla\left(\frac{\eta_r-\eta_0}{s^{1/2}}\right)\right\rvert^2dx+\int_{\Omega_r}\left\lvert \frac{\eta_r-\eta_0}{s^{1/2}}\right\rvert^2dx\\&-&
\frac{1}{s}\int_{\Omega_r}\left\lbrace\nabla(\eta_r-\eta_0)\nabla(p_r-p_0) dx-k^2(\eta_r-\eta_0)(p_r-p_0)\right\rbrace dx
\\
\label{rectifD}&=&\int_{\Omega_r}\left\lvert \nabla\left(\frac{\eta_r-\eta_0}{s^{1/2}}\right)\right\rvert^2dx+\int_{\Omega_r}\left\lvert \frac{\eta_r-\eta_0}{s^{1/2}}\right\rvert^2dx-
\frac{1}{s^{1/2}}\int_{\partial E_r\cap \partial\Omega_r^m}\nabla p_0\cdot\nabla d_E\left(\frac{\eta_r-\eta_0}{s^{1/2}}\right) dx.
\end{eqnarray}
We therefore have the following theorem
\begin{theorem}
Let $E$ verify \textbf{Hypothesis1} and $s=\alpha_{N-d}r^{N-d}$.\\
   The topological derivative (\ref{delderiveD}) exists if and only if the following limit exists
   \begin{eqnarray*}
   l=\lim_{ r\searrow 0}\;(l_0(r)+l_1(r));\qquad l_0(r)&=& \int_{\Omega_r}\left\lvert \nabla\left(\frac{\eta_r-\eta_0}{s^{1/2}}\right)\right\rvert^2dx+\int_{\Omega_r}\left\lvert \frac{\eta_r-\eta_0}{s^{1/2}}\right\rvert^2dx\\
   l_1(r)&=&
   \frac{1}{s^{1/2}}\int_{\partial E_r\cap \partial\Omega_r^m}\nabla p_0\cdot\nabla d_E\left(\frac{\eta_r-\eta_0}{s^{1/2}}\right) dx.
   \end{eqnarray*}
   The topological derivative is given by the expression
   \begin{eqnarray*}
   dJ
      l-\int_E\left( \lvert\nabla\eta_0-A\rvert^2+\lvert\eta_0-\eta_d\rvert^2+\nabla\eta_0\nabla p_0-k^2\eta_0p_0-fp_0\right)dH^d 
   \end{eqnarray*}
   where $(\eta_0,p_0)$ is the solution of the coupled system (\ref{etateta0D})-(\ref{etatp0D}).
  \\In particular for $d=0$;
     \begin{eqnarray*}
     dJ=l-\lvert\nabla\eta_0(x_0)-A(x_0)\rvert^2-\lvert\eta_0(x_0)-\eta_d(x_0)\rvert^2-\nabla\eta_0(x_0)\nabla p_0(x_0) +k^2\eta_0(x_0)p_0(x_0)+f(x_0)p_0(x_0).
     \end{eqnarray*}
   \end{theorem}

   \subsection{Neumann condition on the boundary of the hole}
   \noindent
   In the case where we put a Neumann condition on the boundary of the hole, the perturbed problem becomes
      \begin{equation}\label{Neudelf}
      \begin{cases}
      -\Delta\eta_r-k^2\eta_r=f\qquad\text{in}\;\;\Omega_r\\[0.2cm]
      \qquad\qquad\;\;\, \eta_r=0\qquad\text{on}\;\;\partial\Omega\\[0.2cm]
      \qquad\qquad \frac{\partial\eta_r}{\partial n}=0\qquad\text{on}\;\;\partial E_r
      \end{cases}
      \end{equation}
      First of all, we note that the calculations remain the same. The small changes that exist are simply in the spaces considered. The techniques for calculating the derivative with respect to the $r$-Lagrangian variables remain unchanged. So in the following the details of the calculations will be omitted. \\ We construct an extension to $\Omega$ by always introducing the solution $\eta^0_r:E_r^0 \rightarrow\mathbb{R}$ of the problem
      \begin{eqnarray*}
      -\Delta\eta^0_r-k^2\eta^0_r=f\;\;\text{in}\;\;E^0_r,\;\;\eta^0_r=\eta_r\;\text{on}\;\partial E_r.
      \end{eqnarray*}
      Always the hypothesis $\Omega_r$ has two components $\Omega_r^m$ and $\Omega_r^0$ is assumed. $\Omega_r^m$ is the component of $\Omega_r$ for which $\partial\Omega$ is part of its boundary. $\Omega_r^0$ is the blind component of $\Omega_r$ whose boundary has an empty intersection with $\partial\Omega$. The function $\eta_r$ is distributed between the two components $\Omega_r^0$ and $\Omega_r^m$ as
      \begin{eqnarray*}
      \eta_r=\eta_0 \;\;\text{in}\;\;\Omega_r^0\;\;\;\text{and}\;\;\; -\Delta\eta_r-k^2\eta_r=f \;\;\text{in}\;\;  \Omega_r^m,
      \end{eqnarray*}
      \begin{eqnarray*}
      \begin{cases}
      \eta_r=0\;\;\text{on}\;\;\partial\Omega\\[0.2cm] \eta_r=0\;\;\text{on}\;\;\partial\Omega_r^m\cap\partial E_r.
      \end{cases}
      \end{eqnarray*}
      Since $\partial E_r$ is made up of two disjoint edges $\partial\Omega_r^0$ and $\partial\Omega_r^m\cap\partial E_r$, we can construct an extension to $\Omega$ by defining the solution $\eta^0_r: E_r^0\rightarrow\mathbb{R}$ of the problem
      \begin{eqnarray*}
      -\Delta\eta^0_r-k^2\eta^0_r=f \;\;\text{in}\;\;  E_r^0,\;\;\;\;\eta_r^0&=&\eta_r\;\;\text{on}\;\;\partial\Omega_r^m\cap\partial E_r\\\eta_r^0&=&\eta_r\;\;\text{on}\;\;\partial\Omega_r^0.
      \end{eqnarray*}
      For simplicity, $\Omega_r^0$ is always assumed to be empty. The solution space of the perturbed problem to consider is
      \begin{eqnarray*}
      \mathcal{U}_r=\left\{\eta\in H^1(\Omega_r):\;\;\eta=0\;\text{on}\;\partial\Omega \right\}.
      \end{eqnarray*}
      The variational formulation of (\ref{Neudelf}) consists in finding $\eta_r\in \mathcal{U}_r$ such that
      \begin{eqnarray*}
      -\int_{\partial\Omega}\frac{\partial\eta_r}{\partial n}vd\sigma-\int_{\partial E_r}\frac{\partial\eta_r}{\partial n}vd\sigma+\int_{\Omega_r}\nabla\eta_r\nabla vdx-\int_{\Omega_r}k^2\eta_r vdx-\int_{\Omega_r}fvdx=0,\;\;\text{for all}\;\;v\in\mathcal{\tilde{U}}_r.
      \end{eqnarray*}
    According to the equation (\ref{Neudelf}), the integral on the boundary of $E_r$ is zero. It will therefore suffice to take the set of test functions, 
      \begin{eqnarray*}
         \mathcal{\tilde{U}}_r=\left\{v\in H^1(\Omega_r):\;\;v=0\;\;\;\text{on}\;\partial\Omega \right\}= \mathcal{U}_r,
         \end{eqnarray*} 
    in order to get
      \begin{eqnarray*}
      \int_{\Omega_r}\nabla\eta_r\nabla vdx-\int_{\Omega_r}k^2\eta_r vdx-\int_{\Omega_r}fv dx=0,\;\;\text{for all}\;\;v\in \mathcal{\tilde{U}}_r.
      \end{eqnarray*}
      The $r$-dependent Lagrangian is then a function defined from $[0,R]\times\mathcal{U}_r$ to values in $\mathbb{R}$ by
       \begin{eqnarray*}
       L(r,\varphi,\phi)=\int_{\Omega_r}\lvert\nabla\varphi-A\rvert^2dx+\int_{\Omega_r}\lvert\varphi-\eta_d\rvert^2dx+\int_{\Omega_r}\nabla\varphi\nabla \phi dx-\int_{\Omega_r}k^2\varphi\phi dx-\int_{\Omega_r}f\phi dx
       \end{eqnarray*}
       \begin{eqnarray*}
       J(\Omega_r)=\inf_{\varphi\in\mathcal{U}_r}\sup_{\phi\in\mathcal{U}_r} L(s,\varphi,\phi)=\int_{\Omega_r}\left(\lvert\nabla\eta_r-A\rvert^2+\lvert\eta_r-\eta_d\rvert^2+\nabla\eta_r\nabla p_0-k^2\eta_r p_0-fp_0\right) dx
       \end{eqnarray*}
       \begin{eqnarray*}
        J(\Omega)=\inf_{\varphi\in H^1_0(\Omega)}\sup_{\phi\in H^1_0(\Omega)} L(0,\varphi,\phi)=\int_{\Omega}\left(\lvert\nabla\eta_0-A\rvert^2+\lvert\eta_0-\eta_d\rvert^2+\nabla\eta_0\nabla p_0-k^2\eta_0 p_0-fp_0\right) dx
        \end{eqnarray*}
   The derivative of the $r$-dependent Lagrangian with respect to $\varphi$ is:
       \begin{eqnarray*}
       d_{\varphi} L(r,\varphi,\phi;\varphi')=\int_{\Omega_r}\left\lbrace 2(\nabla\varphi-A)\nabla\varphi'+2(\varphi-\eta_d)\varphi'\right\rbrace dx+\int_{\Omega_r}\nabla\varphi'\nabla \phi dx-\int_{\Omega_r}k^2\varphi'\phi dx
        \end{eqnarray*}
        The derivative of the $r$-dependent Lagrangian with respect to $phi$ is:
         \begin{eqnarray*}
         d_{\phi} L(s,\varphi,\phi;\phi')=\int_{\Omega_r}\nabla\varphi\nabla \phi' dx-\int_{\Omega_r}k^2\varphi\phi' dx-\int_{\Omega}f\phi' dx.
          \end{eqnarray*}
       For $d=0$, $E_r=\overline{B(x_0,r)}$,
         \begin{eqnarray*}
         d_sL(0,\eta_0,p_0)=-\lvert\nabla\eta(x_0)-A(x_0)\rvert^2-\lvert\eta_0(x_0)-\eta_d(x_0)\rvert^2-\nabla\eta_0(x_0)\nabla p_0(x_0) +k^2\eta_0(x_0)p_0(x_0).
         \end{eqnarray*}
       For the cases where $d=1,2$,  
        \begin{eqnarray*}
           d_sL(0,\eta_0,p_0)=-\int_{E}\left(\lvert\nabla\eta_0-A\rvert^2+\lvert\eta_0-\eta_d\rvert^2+\nabla\eta_0\nabla p_0 -k^2\eta_0p_0\right) dH^d.
           \end{eqnarray*}
      The state and adjoint state $\eta_0$ and $p_0$ are solutions of 
     \begin{eqnarray}\label{etateta0N}
        \text{Find}\,\eta_0\in H^1_0(\Omega),\,\, \int_{\Omega}\nabla\eta_0\nabla \phi' dx-\int_{\Omega}k^2\eta_0\phi' dx=\int_{\Omega} f\phi' dx,\;\forall\,\phi'\in H^1_0(\Omega).
     \end{eqnarray}
     \begin{small}
      \begin{align}\label{etatp0N}
      \text{Find}\,p_0\in H^1_0(\Omega),\; \int_\Omega\left\lbrace 2(\nabla \eta_0-A)\nabla\varphi'+2(\eta_0-\eta_d)\varphi'\right\rbrace dx+\int_{\Omega}\left\lbrace\nabla\varphi'\nabla p_0 -k^2\varphi'p_0\right\rbrace dx=0,\;\forall\,\varphi'\in H^1_0(\Omega).
        \end{align}
         \end{small}
     The function $R(r)$ is defined by    
   \begin{equation}\label{R(s)3}
   R(r)=\int_{\Omega_r}\left\lvert \nabla\left(\frac{\eta_r-\eta_0}{s^{1/2}}\right)\right\rvert^2dx+\int_{\Omega_r}\left\lvert \frac{\eta_r-\eta_0}{s^{1/2}}\right\rvert^2dx-
   \frac{1}{s^{1/2}}\int_{\partial E_r\cap \partial\Omega_r^m}\nabla p_0\cdot\nabla d_E\left(\frac{\eta_r-\eta_0}{s^{1/2}}\right) dx.
   \end{equation}
   We therefore have the following theorem
   \begin{theorem}\label{neutheo}
   Let $0\leq d< N$, $E$ verifying the hypothesis \textbf{H1} and $s=\alpha_{N-d}r^{N-d}$.\\
         The topological derivative exists if and only if the following limit exists
      \begin{eqnarray*}
      l=\lim_{ r\searrow 0}(l_0(r)+l_1(r);\qquad l_0(r)&=& \int_{\Omega_r}\left\lvert \nabla\left(\frac{\eta_r-\eta_0}{s^{1/2}}\right)\right\rvert^2dx+\int_{\Omega_r}\left\lvert \frac{\eta_r-\eta_0}{s^{1/2}}\right\rvert^2dx\\
      l_1(r)&=&
      \frac{1}{s^{1/2}}\int_{\partial E_r\cap \partial\Omega_r^m}\nabla p_0\cdot\nabla d_E\left(\frac{\eta_r-\eta_0}{s^{1/2}}\right) dx.
      \end{eqnarray*}
      Then, in this case, the topological derivative is given by the following expression
      \begin{eqnarray*}
      dJ=
         l-\int_E\left( \lvert\nabla\eta_0-A\rvert^2+\lvert\eta_0-\eta_d\rvert^2+\nabla\eta_0\nabla p_0-k^2\eta_0p_0-fp_0\right)dH^d 
      \end{eqnarray*}
      where $(\eta_0,p_0)$ is the solution of the coupled system (\ref{etateta0N})-(\ref{etatp0N}).
           \\In particular for $d=0$;
           \begin{eqnarray*}
           dJ=l-\lvert\nabla\eta_0(x_0)-A(x_0)\rvert^2-\lvert\eta_0(x_0)-\eta_d(x_0)\rvert^2-\nabla\eta_0(x_0)\nabla p_0(x_0) +k^2\eta_0(x_0)p_0(x_0)+f(x_0)p_0(x_0).
           \end{eqnarray*}
      \end{theorem}

\subsection{Example of the one-dimensional Helmholtz equation}
The calculation of $R(s)$ is not generally easy to obtain. And even if one manages to calculate the limit giving $R(s),$ one is not always reassured of the existence of its limit. Sometimes the limit does not exist, sometimes if it is non-zero, sometimes it is zero. In the following, we give examples of explicit calculations where we  have $l=\infty.$ 
But unfortunately the construction is not automatic, there are many things involved.\\
Let $\Omega=]-1,1[$, $E=\{0\}$ and 
 \begin{eqnarray}\label{exple1}
 \eta_0''+k^2\eta_0=x\;\;\text{in}\;\;]-1,1[ \;\;\text{and}\;\;\eta_0(\pm 1)=0
 \end{eqnarray}
where $\eta_0$ is a function of a variable $x$ and $\eta_0''$ its second derivative with respect to $x$.  
 The general solution of (\ref{exple1}) takes the form
 \begin{eqnarray*}
 \eta_0(x)=A\cos(kx)+B\sin(kx)+\frac{x}{k^2},
 \end{eqnarray*}
  where $A$ and $B$ are constants to be determined from the initial conditions. After the calculations we can have $A=0$ and $B=-\frac{1}{k^2\sin k}$, and in this case we obtain the following solution
  \begin{eqnarray*}
  \eta_0(x)=-\frac{\sin(kx)}{k^2\sin k}+\frac{x}{k^2}.
  \end{eqnarray*}
  We define for $0< r < 1$ respectively $E_r$ and $\Omega_r$ by
    \begin{eqnarray*}
    E_r=\{x\in\mathbb{R},\;\;\lvert x\rvert\leq r  \}=[-r,r]\;\;\text{and}\;\;\Omega_r=]-1,-r[\cup ]r,1[ .
    \end{eqnarray*}
  The perturbed system (for a Dirichlet condition on the boundary of $\partial E_r$) is then defined by 
  \begin{eqnarray*}
  \eta_r''+k^2\eta_r=x\;\;\text{in}\;\;]-1,-r[\cup ]r,1[ \;\;\text{et}\;\;\eta_r(\pm 1)=0\;\;\eta_r(\pm r)=0.
  \end{eqnarray*}
In $]-1,-r[$,
  \begin{eqnarray*}
  \eta_r(x)&=& \left\lbrace \frac{1}{k^2\cos k}+\frac{\tan k}{k^2(\tan k\cos kr-\sin kr)}\left(r-\frac{\cos kr}{\cos k}\right)\right\rbrace\cos kx\\ &+& \left\lbrace \frac{1}{k^2(\tan k\cos kr-\sin kr)}\left(r-\frac{\cos kr}{\cos k}\right)\right\rbrace\sin kx+\frac{x}{k^2}
  \end{eqnarray*}
  and in $]r,1[$,
  \begin{eqnarray*}
    \eta_r(x)&=& \left\lbrace -\frac{1}{k^2\cos k}-\frac{\tan k}{k^2(\sin kr-\tan k\cos kr)}\left(\frac{\cos kr}{\cos k}-r\right)\right\rbrace\cos kx\\ &+& \left\lbrace \frac{1}{k^2(\sin kr-\tan k\cos kr)}\left(\frac{\cos kr}{\cos k}-r\right)\right\rbrace\sin kx+\frac{x}{k^2}.
    \end{eqnarray*}
  As for the extension in $[-r,r]$, we define the function $\eta_r^0$ satisfying
  \begin{eqnarray*}
    (\eta^0_r)''+k^2\eta^0_r=x\;\;\text{in}\;\;[-r,r] \;\;\text{and}\;\;\eta^0_r(\pm r)=\eta^0_r(\pm r)=0
    \end{eqnarray*}
    whose solution is
    \begin{eqnarray*}
    \eta^0_r(x)=-\frac{r}{k^2\sin kr}\sin(kx)+\frac{x}{k^2}.
    \end{eqnarray*}
 The extended function $\eta^r$ in $]-1,1[$  is then defined by:\\
   in $]-1,-r[$,
     \begin{eqnarray*}
     \eta^r(x)&=& \left\lbrace \frac{1}{k^2\cos k}+\frac{\tan k}{k^2(\tan k\cos kr-\sin kr)}\left(r-\frac{\cos kr}{\cos k}\right)\right\rbrace\cos kx\\ &+& \left\lbrace \frac{1}{k^2(\tan k\cos kr-\sin kr)}\left(r-\frac{\cos kr}{\cos k}\right)\right\rbrace\sin kx+\frac{x}{k^2},
     \end{eqnarray*}
  in $[-r,r]$
      \begin{eqnarray*}
         \eta^r(x)=-\frac{r}{k^2\sin kr}\sin(kx)+\frac{x}{k^2}
         \end{eqnarray*}
  and in $]r,1[$,
  \begin{small}
     \begin{eqnarray*}
       \eta^r(x)&=& \left\lbrace -\frac{1}{k^2\cos k}-\frac{\tan k}{k^2(\sin kr-\tan k\cos kr)}\left(\frac{\cos kr}{\cos k}-r\right)\right\rbrace\cos kx\\ &+& \left\lbrace \frac{1}{k^2(\sin kr-\tan k\cos kr)}\left(\frac{\cos kr}{\cos k}-r\right)\right\rbrace\sin kx+\frac{x}{k^2}.
       \end{eqnarray*}
       \end{small}
Let us note by $w^r=\eta^r-\eta_0$, to calculate the limit of $R(r)$, \\
\begin{small}
 in $]-1,-r[$,
     \begin{eqnarray*}
   w^r(x)&=& \left\lbrace \frac{1}{k^2\cos k}+\frac{\tan k}{k^2(\tan k\cos kr-\sin kr)}\left(r-\frac{\cos kr}{\cos k}\right)\right\rbrace\cos kx\\ &+& \left\lbrace \frac{1}{k^2\sin k}+ \frac{1}{k^2(\tan k\cos kr-\sin kr)}\left(r-\frac{\cos kr}{\cos k}\right)\right\rbrace\sin kx,
     \end{eqnarray*}
  in $[-r,r]$
      \begin{eqnarray*}
       w^r(x)=\left( \frac{1}{k^2\sin k} -\frac{r}{k^2\sin kr}\right)\sin(kx)
         \end{eqnarray*}
 and in $]r,1[$,
  \begin{eqnarray*}
  w^r(x)&=& \left\lbrace -\frac{1}{k^2\cos k}-\frac{\tan k}{k^2(\sin kr-\tan k\cos kr)}\left(\frac{\cos kr}{\cos k}-r\right)\right\rbrace\cos kx\\ &+& \left\lbrace \frac{1}{k^2\sin k}+ \frac{1}{k^2(\sin kr-\tan k\cos kr)}\left(\frac{\cos kr}{\cos k}-r\right)\right\rbrace\sin kx
    \end{eqnarray*}
 \end{small}
\begin{small}
   \begin{eqnarray*}
   \left\lVert \frac{w^r}{\sqrt{2r}}\right\rVert^2_{L^2(\Omega_r)}
   &=&\frac{1}{2r} \left\lbrace \frac{1}{k^2\cos k}+\frac{\tan k}{k^2(\tan k\cos kr-\sin kr)}\left(r-\frac{\cos kr}{\cos k}\right)\right\rbrace^2\left( \frac{1}{2}(1-r)+\frac{1}{4}(\sin 2k-\sin 2kr)\right)
   \\ &+&\frac{1}{2r}  \left\lbrace \frac{1}{k^2\sin k}+ \frac{1}{k^2(\tan k\cos kr-\sin kr)}\left(r-\frac{\cos kr}{\cos k}\right)\right\rbrace^2\left( \frac{1}{2}(1-r)-\frac{1}{4}(\sin 2k-\sin 2kr)\right)
   \\&+&\frac{1}{2r}  \left\lbrace -\frac{1}{k^2\cos k}-\frac{\tan k}{k^2(\sin kr-\tan k\cos kr)}\left(\frac{\cos kr}{\cos k}-r\right)\right\rbrace^2\left( \frac{1}{2}(1-r)+\frac{1}{4}(\sin 2k-\sin 2kr)\right)\\ &+&\frac{1}{2r}  \left\lbrace \frac{1}{k^2\sin k}+ \frac{1}{k^2(\sin kr-\tan k\cos kr)}\left(\frac{\cos kr}{\cos k}-r\right)\right\rbrace^2\left( \frac{1}{2}(1-r)-\frac{1}{4}(\sin 2k-\sin 2kr)\right)\\&\;&\rightarrow +\infty,\;\;\text{when}\;\;r\rightarrow 0
   \end{eqnarray*}
   \end{small}
  \begin{eqnarray*}
  \left\lVert \frac{w^r}{\sqrt{2r}}\right\rVert^2_{L^2(E_r)}=\frac{1}{2r}\left( \frac{1}{k^2\sin k} -\frac{r}{k^2\sin kr}\right)^2\left(r-\frac{1}{2}\sin 2kr\right).
  \end{eqnarray*} 
The function R(s) takes the following form:
 \begin{small}
  \begin{eqnarray*}
  R(r)&=&\int_{-1}^{-r}\left\lvert\frac{(w^r)'}{\sqrt{2r}}\right\rvert^2dx+\int_{r}^{1}\left\lvert\frac{(w^r)'}{\sqrt{2r}}\right\rvert^2dx+\int_{-1}^{-r}\left\lvert\frac{w^r}{\sqrt{2r}}\right\rvert^2dx\\&+&\int_{r}^{1}\left\lvert\frac{w^r}{\sqrt{2r}}\right\rvert^2dx+\frac{1}{\sqrt{2r}}\int_{-r}^{r}p_0'd'_E\left(\frac{w^r}{\sqrt{2r}}\right)dx\\[0.2cm]
  &=&\left\lVert \frac{w^r}{\sqrt{2r}}\right\rVert^2_{L^2(\Omega_r)}+\left\lVert \frac{(w^r)'}{\sqrt{2r}}\right\rVert^2_{L^2(\Omega_r)}+\frac{1}{\sqrt{2r}}\int_{-r}^{r}p_0'd'_E\left(\frac{w^r}{\sqrt{2r}}\right)dx.
  \end{eqnarray*} 
  \end{small}
\textbf{Solving the adjoint equation}
In $\Omega=]-1,1[,$ the adjoint equation is solution to
  \begin{eqnarray*}
 - p''_0(x)-k^2p_0(x)=2(\eta''_0-A')+2(\eta_0-\eta_d)\;\;\text{in}\;\;]-1,1[ \;\text{and}\;\;p_0(\pm 1)=0
  \end{eqnarray*} 
 To simplify, let $A$ be a constant function and $\eta_d=\frac{x}{k^2}$. Then by replacing $\eta''_0$ and $\eta_0$ by their expression, the equation becomes
   \begin{eqnarray}\label{p0eq}
     -p''_0(x)-k^2p_0(x)=\frac{2(k^2-1)}{k^2\sin k}\sin kx
     \end{eqnarray} 
  whose characteristic equation $r^2+k^2=0$ has a negative discriminant, so the solutions of the equation without second member are the functions
  $x\rightarrow p_H(x)=\alpha\cos kx+\beta\sin kx$ where $\alpha,\beta\in\mathbb{R}$.\\
In the following, we are looking for a particular solution.
 Let's put $A=\frac{2(k^2-1)}{k^2\sin k}$. Note that here the second member is of the form $AIm(e^{ikx})$, so we are looking for a solution of the form $p_p(x)=Im(Bxe^{ikx})$. Deriving $p_p(x)$ twice and replacing it in the equation gives
   \begin{eqnarray*}
   -Im\left( (2iBk-Bk^2x)e^{ikx}\right)+k^2Im(Bxe^{ikx})=AIm(e^{ikx})\\[0.2cm] \Rightarrow  -Im\left[ (2iBk-Bk^2x+Bk^2x)e^{ikx}\right]=AIm(e^{ikx})\\[0.2cm] \Rightarrow -Im\left[(2iBk)e^{ikx}\right]=AIm(e^{ikx})
   \end{eqnarray*}
 And by identification $-2iBk=A$, which implies that $B=-\frac{A}{2ki}=\frac{iA}{2k}$. So the particular solution is
\begin{eqnarray*}
 p_p(x)=Im\left( \frac{iA}{2k} xe^{ikx}\right)=Im\left( \frac{iA}{2k}x\cos kx-\frac{A}{2k}x\sin kx\right)=\frac{A}{2k}x\cos kx,
 \end{eqnarray*}
  and the general solution $p_0$ is 
  \begin{eqnarray*}
  p_0(x)=\alpha\cos kx+\beta\sin kx+\frac{A}{2k}x\cos kx
  \end{eqnarray*}

  \begin{eqnarray}
   \label{p0-1} p_0(-1)=0&\Rightarrow&\alpha\cos k-\beta\sin k-\frac{A}{2k}\cos k=0\\ \label{p01} p_0(1)=0&\Rightarrow&\alpha\cos k+\beta\sin k+\frac{A}{2k}\cos k=0
    \end{eqnarray}
    According to equations (\ref{p0-1}) and (\ref{p01}), $\alpha=0$ and $\beta=-\frac{A}{2k\sin k}\cos k=-\frac{A}{2k\tan k}$. 
    The general solution $p_0$ is then
   \begin{eqnarray*}
         p_0(x)=-\frac{A}{2k\tan k}\sin kx+\frac{A}{2k}x\cos kx\;\;\text{and}\;\;p'_0(x)=\frac{A}{2k}\cos kx-\frac{A}{2\tan k}\cos kx-\frac{A}{2}x\sin kx
         \end{eqnarray*}
    The derivative of the function $w_r$, $(w_r)'$  is also given as\\
   in $]-1,-r[$,
        \begin{eqnarray*}
      (w_r)'(x)&=& \left\lbrace -\frac{1}{k\cos k}-\frac{\tan k}{k(\tan k\cos kr-\sin kr)}\left(r-\frac{\cos kr}{\cos k}\right)\right\rbrace\sin kx\\ &+& \left\lbrace \frac{1}{k\sin k}+ \frac{1}{k(\tan k\cos kr-\sin kr)}\left(r-\frac{\cos kr}{\cos k}\right)\right\rbrace\cos kx,
        \end{eqnarray*}
      in $[-r,r]$
         \begin{eqnarray*}
          (w_r)'(x)=\left( \frac{1}{k\sin k} -\frac{r}{k\sin kr}\right)\cos(kx)
            \end{eqnarray*}
      and in $]r,1[$,
     \begin{eqnarray*}
     (w_r)'(x)&=& \left\lbrace \frac{1}{k\cos k}+\frac{\tan k}{k(\sin kr-\tan k\cos kr)}\left(\frac{\cos kr}{\cos k}-r\right)\right\rbrace\sin kx\\ &+& \left\lbrace \frac{1}{k\sin k}+ \frac{1}{k(\sin kr-\tan k\cos kr)}\left(\frac{\cos kr}{\cos k}-r\right)\right\rbrace\cos kx
       \end{eqnarray*}
   \begin{small}
    \begin{eqnarray*}
        \left\lVert \frac{(w^r)'}{\sqrt{2r}}\right\rVert^2_{L^2(\Omega_r)}&=& \frac{1}{2r}\left\lbrace -\frac{1}{k\cos k}-\frac{\tan k}{k(\tan k\cos kr-\sin kr)}\left(r-\frac{\cos kr}{\cos k}\right)\right\rbrace^2\left( \frac{1}{2}(1-r)-\frac{1}{4}(\sin 2k-\sin 2kr)\right)\\ &+& \frac{1}{2r}\left\lbrace \frac{1}{k\sin k}+ \frac{1}{k(\tan k\cos kr-\sin kr)}\left(r-\frac{\cos kr}{\cos k}\right)\right\rbrace^2\left( \frac{1}{2}(1-r)+\frac{1}{4}(\sin 2k-\sin 2kr)\right)\\
        &+& \frac{1}{2r}\left\lbrace \frac{1}{k\cos k}+\frac{\tan k}{k(\sin kr-\tan k\cos kr)}\left(\frac{\cos kr}{\cos k}-r\right)\right\rbrace^2\left( \frac{1}{2}(1-r)-\frac{1}{4}(\sin 2k-\sin 2kr)\right)\\ &+& \frac{1}{2r}\left\lbrace \frac{1}{k\sin k}+ \frac{1}{k(\sin kr-\tan k\cos kr)}\left(\frac{\cos kr}{\cos k}-r\right)\right\rbrace^2\left( \frac{1}{2}(1-r)+\frac{1}{4}(\sin 2k-\sin 2kr)\right)\\&\;&\rightarrow 0,\;\;\text{when}\;\;r\rightarrow 0
          \end{eqnarray*}
   \end{small}   
    \begin{small}
    \begin{eqnarray*}
    l_1(r)&=&-\frac{1}{\sqrt{2r}}\int_{\partial E_r}p_0'd'_E\left(\frac{w^r(x)}{\sqrt{2r}}\right)dH^0=-\frac{1}{\sqrt{2r}}p'(x)\frac{x}{\lvert x\rvert}\frac{w^r(x)}{\sqrt{2r}}\Big|_{x=-r}^r\\[0.2cm] &=&-\frac{1}{\sqrt{2r}}p'(r)\frac{r}{\lvert r\rvert}\frac{w^r(r)}{\sqrt{2r}} +\frac{1}{\sqrt{2r}}p'(-r)\frac{(-r)}{\lvert r\rvert}\frac{w^r(-r)}{\sqrt{2r}}=-\frac{1}{\sqrt{2r}}p'(r)\frac{w^r(r)}{\sqrt{2r}} -\frac{1}{\sqrt{2r}}p'(-r)\frac{w^r(-r)}{\sqrt{2r}}
    \\[0.2cm] &=& -\frac{\sin(kr)}{2r}\left(\frac{A}{2k}\cos kr-\frac{A}{2\tan k}\cos kr-\frac{A}{2}r\sin kr \right) \left( \frac{1}{k^2\sin k} -\frac{r}{k^2\sin kr}\right)
     \\[0.2cm] &\;\;& +\frac{\sin(kr)}{2r}\left(\frac{A}{2k}\cos kr-\frac{A}{2\tan k}\cos kr-\frac{A}{2}r\sin kr \right) \left( \frac{1}{k^2\sin k} -\frac{r}{k^2\sin kr}\right)=0
    \end{eqnarray*}
    \end{small}
In summary, from all the above, we have:  $\displaystyle \lim_{s \rightarrow 0}R(r)= \infty.$ \\
 From this example, considering the  min max approach developed by Delfour, in order to have the existence of a topological sensitivity which is a semi-differential, we will be led to restrict ourselves to specific cases to hope to calculate this topological semi-differential.

\section{ Homogenization with two scale convergence and topological semi differential}
We propose in this section a method allowing to pronounce on the limit of $R(s)$ when s tends to zero. Here we are looking for sufficient conditions to be able to calculate the limit of $R(s).$ Since this limit may exist or not, on is wondering if it is possible to propose a family of boundary value problems for which  it will be quite possible to get the limit of $R(s)$ when $s$ goes to $0.$\\
It is important to underline that in topological optimization, the perturbation  of domain introduces singular situations on the behavior of the state of the system, the most important object in the analysis, the so-called polarization matrix is to be  fully understood. The polarization matrices are present, either in explicit forms or implicit in all asymptotic formulae. For more details about this notion, see for instance \cite{NaSo}, \cite{AmKa}, \cite{TD2} and references therein.
 The method proposed here, is based on an asymptotic analysis approach and more particularly on two-scale convergence. For this method, we can therefore refer to the work of Allaire \cite{Allaire}, Nguetseng \cite{Nguetseng}, Fr\'enod et al. \cite{FRS:1999}, Faye et al.\cite{FaFreSe, FaFrSe1}. One can always show, depending on the field of study, the existence of a second scale either in relation to space or in relation to time. For this, we assume that the solution $\eta^s$ of the perturbed problem admits an asymptotic expansion of the form

\begin{equation}\label{asympt1}\eta^s(x)=\eta_0(x,\frac xs)+s\eta_1(x,\frac xs)+s^2\eta_2(x,\frac xs)+\ldots\end{equation}
where the function 
$\eta_i,\,\,i=0,1,\ldots$ are regular and periodic functions with respect to the second variable. Many characterization results of the $\eta_i(x,\frac xs)$ exist in the literature. One can quote, for example, the work of Allaire \cite{Allaire} and Ngetseng \cite{Nguetseng} in the case of elliptic equations, those of Faye et al \cite{FaFreSe, FaFrSe1} in the case of the parabolic equations and the work of Fr\'enod et al. \cite{FRS:1999} in the case of hyperbolic equations; see also \cite{BaSe}, where parabolic and hyperbolic equations are coupled.  We are going to  prove, in this paper, that if any sequence $\eta^s$ verifying an asymptotic expansion of the type (\ref{asympt1}) where the functions $\eta_i(x, y)$ are smooth
and $Y$-periodic in $y,$  and two-scale converges to the first term of the expansion, namely,
$\eta_0(x, y)$ in a sense that we shall precise. In many cases, $Y$ is choseen to be equals to $[0,1]^N,\Omega\subset R^N.$ In addition, we have:

\begin{equation}\label{asympt2} \eta^s(x)-\eta_0(x,\frac{x}{s})-s\eta_1(x, \frac xs)\longrightarrow 0\,\, \text{in}\,\,H^1(\Omega)\,\,\text{strongly}.
\end{equation}

Let $\eta_0$ be the two scale limit of $\eta^s$ and $\eta_0^s(x)=\eta_0(x, \frac xs).$
The proof of the main result of this part is based on the existence of a corrector term of the type (\ref{asympt2}), i.e.    $$ \frac{\eta^s-\eta_0^s}{s}\rightarrow \eta_1\,\,\text{in}\,\,H^1(\Omega). $$ So we have the following theorem
\begin{theorem} Considering $\eta^s$ the solution to (\ref{DFEqHelmpert0}) and $\eta_0^s=\eta_0^s(x)=\eta_0(x,\frac xs)$ where  $\eta_0$ is the two scale limits of $\eta^s,$ the following  estimate holds

$$\vert\vert\frac{\eta^s-\eta_0^s}{s} \vert\vert_{H^1(\Omega)}
\leq \alpha,$$
where $\alpha$ is a constant not depending on $s.$
Furthermore, the sequence $\frac{\eta^s-\eta_0^s}{s}$ two scale converges to a profile  $\eta_1\in H^1(\Omega, L^2_{\#}(Y)).$
\end{theorem} 
\begin{proof}The sequence $\eta^s$ is solution to the following equation
\begin{equation}\label{ref}\Delta \eta^s+\eta^s=f\,\,\text{in}\,\, \Omega\\
\end{equation}
Let us first characterized the equation satisfied  by the two scale limits.\\

Defining a test function as follows, $\psi^{s}(x)=\psi(x,\frac{x}{s})$ for any $\psi(x, y)$,
regular with compact support in $Y\times\torus^{2}$ and periodic in $y$ with period 1,
multiplying (\ref{ref}) by $\psi^{s}$  and integrating in $\Omega,$ we have

\begin{equation}\label{H3}\int_{\Omega}(\Delta \eta^s+\eta^s)\psi^s dx=\int_{\Omega}f\psi^{s}dtdx.
\end{equation}
Then using Green formula  in the first integral over $\Omega$ and the boundary condition, we get  
\begin{equation}\label{H4} 
\int_{\Omega}\big(-\nabla \eta^s\nabla \psi^s+ \eta^s\psi^s\big)dx=\int_{\Omega}f\psi^s dx
\end{equation}
But
\begin{equation}\label{H5}
 \nabla\psi^s=\left(\nabla_x \psi\right)^{s} +\frac1s\left(\nabla_y\psi\right)^{s},
\end{equation}
where
\begin{equation}\label{H4a}
  \left(\nabla_x\psi\right)^{s}(x)
   =\nabla_x\psi(x,\frac{x}{s})\,\,\textrm{ and }  \left(\nabla_y\psi\right)^{s}(x)=\nabla_y\psi(x,\frac xs),
\end{equation}
then we have
\begin{equation}\label{H6}
\int_{\Omega}\big(-\nabla \eta^s\Big(\left(\nabla_x \psi\right)^{s} +\frac1s\left(\nabla_y\psi\right)^s\Big)+ \eta^s\psi^s\big)dx=\int_{\Omega}f\psi^s dx
\end{equation}
Using the two-scale convergence due to Nguetseng \cite{Nguetseng} and
Allaire \cite{Allaire} (see also Fr\'enod Raviart and Sonnendr\"{u}cker \cite{FRS:1999}),  if a sequence $f^{s}$ is bounded in $L^{2}(\Omega)$, then there exists a profile
$U(x,y)$, periodic of period 1 with respect to $y$, such that for all
$\psi(x,y)\in \mathcal D(\Omega, \mathcal C^\infty(Y)),$ we have
 \begin{equation}\label{H7}
    \int_{\Omega}\int_{Y}f^{s}\psi^{s}dx
   \longrightarrow\int_{\Omega}\int_{Y}U\psi dydx,
 \end{equation}
 for a subsequence extracted from $(f^{s})$.
 \\
Multiplying (\ref{H6}) by $s$ and passing to the limit as $s\rightarrow 0$ and using (\ref{H7})
we have
    \begin{equation}\label{H8}
\int_{\Omega}\int_Y\Big(-\nabla \eta_0\nabla_x\psi+\eta_0\psi\Big)dydx =\int_{\Omega}\int_Y f\psi dydx,
    \end{equation} and
    \begin{equation}\label{H08}
    \int_\Omega \int_Ydiv_y(\nabla_x\eta_0)\psi dx\,dy=0.
    \end{equation}
We obtain from (\ref{H8}) the equation satisfied by $\eta_0$:
\begin{equation}\label{H10}\left\{\begin{array}{ccc}
    \Delta\eta_0+\eta_0 =f\,\,\text{in}\,\,\Omega\times Y\\
    \eta_0=0\,\,\text{on}\,\,\partial\Omega\\
    \nabla_y\cdot\nabla_x\eta_0=0\,\,\text{in}\,\,\Omega\times Y\end{array}\right.
    \end{equation}
For  the following, let
\begin{equation}
\eta^s_0(x)=\eta_0(x,\frac xs),
\end{equation} where $\eta_0$ is solution to (\ref{H10}). Then we have
\begin{equation}\nabla \eta_0^s= \left(\nabla_x \eta_0^s\right)^{s} +\frac1s\left(\nabla_y\eta_0^s\right)^{s},
\end{equation}
where
\begin{equation}\label{H4a0}
  \left(\nabla_x\eta_0^s\right)^{s}(x)
   =\nabla_x\eta_0^s(x,\frac{x}{s})\,\,\textrm{ and }  \left(\nabla_y\eta_0^s\right)^{s}(x)=\nabla_y\eta_0^s(x,\frac xs),
\end{equation}
It follows from a result of Allaire \cite{Allaire}, that the sequence $\eta^s$ two scale converge to a function $\eta_(x)=\int_Y\eta_0(x,y)dy$ and the sequence 
$$\eta^s(x)-\eta_0(x)-s\eta_1(x, \frac xs)\rightarrow 0\,,\,\text{in}\,\, H^1(\Omega)\,\,\text{strongly}$$ and the 
$$\nabla \eta^s(x)-\nabla \eta_0(x)-\nabla_y\eta_1(x, \frac xs) $$ converges strongly to zero in 
$(L^2(\Omega))^N.$ This  result prove that  the sequence $\frac{\eta^s-\eta_0^s}{s}$ is bounded and two scale converges to a profile $\eta_1\in H^1(\Omega, L^2_{\#}(Y)).$\\
In an other hand, $\frac{\eta^s-\eta_0^s}{s}$ is solution to 
\begin{equation} 
\Delta\frac{\eta^s-\eta_0^s}{s}+\frac{\eta^s-\eta_0^s}{s}=\nabla_x\cdot(\nabla_y \eta_0)
\end{equation}
The sequence  $\frac{\eta^s-\eta_0^s}{s}$ is bounded in $H^1(\Omega)$ then, from a result of Allaire \cite{Allaire} and Nguetseng \cite{Nguetseng}, it two scale converges to a profile $\eta_1$ solution to 

\begin{equation}\left\{\begin{array}{ccc}
\Delta\eta_1+\eta_1= \nabla_x\cdot(\nabla_y \eta_0)=0\,\,\text{in}\,\,\Omega\times Y\\
\eta_1=0\,\,\partial\Omega,\\
y\rightarrow \eta_1 \,\,\text{is 1 periodic}\end{array}\right.\end{equation}


\end{proof}

\begin{theorem} Under asymptions (\ref{asympt1}),(\ref{asympt2}) and Hypotheses \textbf{(H0), (H1), (H2), (H3)}, If $R(s)$ is given by 
(\ref{rectifD}) and (\ref{R(s)3}); the limits of $R(s)$ exists et tends to zero as $s\rightarrow 0.$\\
In addition, the functional  defined by (\ref{fye}) is differentiable and the topoogical derivative is given by
\begin{equation}\label{derivv}DJ(\Omega)=\frac{d}{ds}L(0, \eta^0, p^0)\end{equation}
where  $L$ designates the Lagrangian, $\eta^0$ and $p^0$  are solutions to the direct and adjoint states.
\end{theorem}
\begin{proof}
The proof is identical for the two cases considered. Without loss of generality, we will only consider the case where $R(s)$ is given by (\ref{rectifD}), i.e.

$$R(r)=\int_{\Omega_r}\left\lvert \nabla\left(\frac{\eta_r-\eta_0}{s^{1/2}}\right)\right\rvert^2dx+\int_{\Omega_r}\left\lvert \frac{\eta_r-\eta_0}{s^{1/2}}\right\rvert^2dx-
\frac{1}{s^{1/2}}\int_{\partial E_r\cap \partial\Omega_r^m}\nabla p_0\cdot\nabla d_E\left(\frac{\eta_r-\eta_0}{s^{1/2}}\right) dx.$$
To prove that $R(r)$ tends to zero, it is sufficient to show that the expressions of $l_1(r)$ and $l_2(r)$ given in Theorem \ref{neutheo} tend to zero when $r\rightarrow 0.$

We are therefore going to focus on the convergence of each of the terms of the expression of $R(r).$ From (\ref{asympt1}) and  (\ref{asympt2}), we get that; if $\eta^0(x)=\eta_0(x,\frac xs),$ the sequence $\frac{\eta^s-\eta^0}{s}$ to scale converges to a profile $\eta^1\in L^\infty_{\#}((\mathbb R, L^2(\Omega)).$ This makes possible to get the following result
\begin{equation}\label{correct}\vert\vert \frac{\eta^s-\eta_0}{s} \vert\vert_{H^1(\Omega)}\leq\alpha\end{equation}
where $\alpha$ is a constant not depending on $s.$\\
The parameters $s$ and $r$ are proportional, when $r$ goes to zero, $s$ also goes to zero. Instead of $\eta_r$ we can write $\eta_s$. So using the fact that, for small $s,\,\,\vert\vert \frac{\eta^s-\eta_0}{\sqrt s} \vert\vert_{H^1(\Omega)}\leq \vert\vert \frac{\eta^s-\eta_0}{s} \vert\vert_{H^1(\Omega)}\leq\alpha$ we have

$$l_1(r)=l'_1(s)=\sqrt s\int_{\Omega}\chi_{\Omega_r}\left\lvert \nabla\left(\frac{\eta_s-\eta_0}{s}\right)\right\rvert^2dx+\sqrt s\int_{\Omega}\chi_{\Omega_r}\left\lvert \frac{\eta_s-\eta_0}{s}\right\rvert^2dx$$
Taking into account (\ref{correct}), we get that as $\lim_{r\rightarrow 0}l_1(r)=0.$ For the same reasons, the expression given by $l_2(r)$ will also cancel. This means that the limit of $R(r)$ will tend to 0.
For the proof of the last result, it is only necessary to note that $R(\eta^0,p^0)=\lim_{r\rightarrow0}R(r)=0.$  By application of Sturm's and Delfour's theorem, the shape derivative or the topological derivative are nothing else than the derivative of the Lagrangian with respect to the variable $s.$ Then (\ref{derivv}) is true.
\end{proof}
\subsection*{ACKNOWLEDGEMENT}
This work has been supported by the Deutsche Forschungsgemeinschaft
within the Priority program SPP 1962 "Non-smooth and Complementarity based Distributed Parameter Systems: Simulation and Hierarchical Optimization". The authors would like to thank Volker Schulz ( University Trier, Trier, Germany) and Luka Schlegel (University Trier, Trier, Germany) for helpful and interesting discussions within the project Shape Optimization Mitigating Coastal Erosion (SOMICE).

\end{document}